\documentclass[10pt]{amsart}

\usepackage{amsfonts,amsmath,latexsym,amssymb,verbatim,amsbsy,amsthm}
\usepackage{dsfont,bm}
\usepackage{graphicx}
\usepackage{multirow}

\usepackage[top=1in, bottom=1in, left=1in, right=1in]{geometry}

\usepackage[dvipsnames]{xcolor}

\usepackage[colorlinks=true, pdfstartview=FitV, linkcolor=RoyalBlue,citecolor=ForestGreen, urlcolor=blue]{hyperref}

\theoremstyle{plain}
\newtheorem{THEOREM}{Theorem}[section]

\newtheorem{theorem}[THEOREM]{Theorem}

\theoremstyle{definition}

\theoremstyle{remark}

\newtheorem{remark}[THEOREM]{Remark}

\newcommand{\thm}[1]{Theorem~\ref{#1}}

\newcommand{\sect}[1]{Section~\ref{#1}}

\def \a {\alpha}

\def \g {\gamma}
\def \d {\delta}
\def \e {\varepsilon}

\def \l {\lambda}
\def \n {\nabla}
\def \s {\sigma}

\def \th {\theta}
\def \o {\omega}
\def \w {\omega}

\def \D {\Delta}

\def \O {\Omega}

\def \bk {{\bf k}}
\def \bl {{\bf l}}

 \def \st {\mathrm{s}}

 \def \en {\mathrm{e}}

\def \cD {\mathcal{D}}
\def \cE {\mathcal{E}}

\def \cH {\mathcal{H}}
\def \cI {\mathcal{I}}

\def \cP {\mathcal{P}}
\def \cQ {\mathcal{Q}}
\def \cR {\mathcal{R}}
\def \cS {\mathcal{S}}

\newcommand{\N}{\ensuremath{\mathbb{N}}}   %%% naturals
\newcommand{\Z}{\ensuremath{\mathbb{Z}}}   %%% integers
\newcommand{\R}{\ensuremath{\mathbb{R}}}   %%% reals
\newcommand{\T}{\ensuremath{\mathbb{T}}}   %%% torus
\newcommand{\E}{\ensuremath{\mathbb{E}}}   %%% expectation
\def \const {\mathrm{const}}

\def \Lip {\mathrm{Lip}}

\newcommand{\jap}[1]{\left\langle #1 \right\rangle}
\newcommand{\ave}[1]{ \left[ #1 \right]}

\DeclareMathOperator{\supp}{supp} %

\def \lan {\langle}
\def \ran {\rangle}

\def \p {\partial}

\def \ss {\subset}

\def \GL {Gr\"onwall's Lemma}

\def \CK{Csisz\'ar-Kullback inequality}

\renewcommand{\geq}{\geqslant}

\renewcommand{\leq}{\leqslant}

\def \dg  {\, \mbox{d}\gamma}
\def \dx  {\, \mbox{d}x}

\def \dxi  {\, \mbox{d}\xi}

\def \dr  {\, \mbox{d}r}
\def \domega  {\, \mbox{d}\omega}
\def \ds  {\, \mbox{d}s}

\def \dth  {\, \mbox{d}\th}

\def \domega  {\, \mbox{d}\omega}

\def \dv  {\, \mbox{d} v}

\def \ddt  {\frac{\mbox{d\,\,}}{\mbox{d}t}}

\def \domain {{\O \times \R^n}}

\begin{document}

\title[Modulation analysis]{Modulation of the monokinetic limit for models of collective dynamics}

\author{Alina Chertock}

\address{Department of Mathematics, North Carolina State University, Raleigh, NC, USA}

\email{acherto@ncsu.edu}

\author{Roman Shvydkoy}

\address{851 S Morgan St, M/C 249, Department of Mathematics, Statistics and Computer Science, University of Illinois at Chicago, Chicago, IL 60607}

\email{shvydkoy@uic.edu}

\author{Trevor Teolis}

\address{851 S Morgan St, M/C 249, Department of Mathematics, Statistics and Computer Science, University of Illinois at Chicago, Chicago, IL 60607}

\email{tteoli2@uic.edu}


\date{\today}


\thanks{\textbf{Acknowledgment.}  
The work of A. Chertock was supported by NSF grant DMS-2208438. The work of R. Shvydkoy was supported in part by NSF
grant  DMS-2405326 and the Simons Foundation. }

\begin{abstract} In this work, we perform modulation analysis of monokinetic limits from the kinetic Cucker-Smale model to the pressureless 
Euler alignment system. Two regimes are considered -- a strong Fokker-Planck force with vanishing noise and Knudsen number, and a pure 
noiseless Vlasov scheme. In the former case, we demonstrate convergence of the modulated profile to the standard Gaussian distribution, 
while in the latter case, the distribution converges to a profile satisfying an explicit transport equation along limiting characteristics.
\end{abstract}

\maketitle 
\section{Introduction}
One of the central models of collective dynamics that has undergone  numerous studies recently is the hydrodynamic pressureless 
Euler-Alignment system given by 
\begin{equation}\label{e:EAS}
\begin{split}
\p_t \rho+ \n\cdot ({u}{\rho}) & = 0, \\
\p_t u + u \cdot \n {u} &  = (u \rho)_\phi - u \rho_\phi.
\end{split}
\end{equation}
Here, $\rho$ is the density of the flock, $u$ its velocity, and $\phi$ is a smooth radially symmetric communication kernel, and we denote 
$\rho_\phi=\rho\ast\phi$ and $(u\rho)_\phi=(u\rho)\ast\phi$. Most commonly, the system is considered on the open space $\O=\R^n$ or 
periodic environment $\O = \T^n$. 

The system models a monokinetic closure of the discrete classical Cucker-Smale alignment system introduced in \cite{CS2007a,CS2007b}:
\begin{equation}\label{e:CS}
\dot{x}_i = v_i, \quad \dot{v}_i =\sum_{j=1}^N m_j \phi(x_i-x_j) (v_j - v_i),
\end{equation}
where a group of $N$ agents with $i$-th agent being represented by its position $x_i$ and velocity $v_i,\ i=1,\ldots,N$. We refer the 
reader to the surveys \cite{ABFHKPPS,Axel97,Ben2005,Edel2001,Darwin,VZ2012,MT2014,MP2018,S-book,Tadmor-notices} for detailed mathematical 
accounts and examples of applications to swarming, satellite navigation, control, etc. 

The kinetic description of the large crowd limit as $N\to 0$ was introduced by Ha and Tadmor in \cite{HT2008} and justified through the 
mean-field limit in \cite{HL2009}. The corresponding Vlasov-alignment equation for a probability distribution of a flock  $f$ is given by
\begin{equation}\label{e:VA}
\p_t f + v \cdot \n_x f =  \n_v( (\rho_\phi v - (u\rho)_\phi ) f ), \quad x\in \O, v\in \R^n.
\end{equation}
Thanks to the maximum principle on the $v$-support of $f$, the equation also satisfies the conditions for  Dobrushin's stability theorem 
stating that for any two measure-valued solutions $\mu',\mu''$ of \eqref{e:VA}, one has
\[
W_1(\mu'(t),\mu''(t)) \leq Ce^{L t}  W_1(\mu'(0),\mu''(0)),
\]
where generally $C,L>0$, and $L=0$ if the kernel $\phi$ is fat-tail, $\int_0^\infty \phi(r) \dr = \infty$, and $W_1$ is the 
Kantorovich-Rubinstein metric (or of Wasserstein-1 distance), see \cite{CCR2011,Ha-stability, S-book}.
It can be formally checked that $(\rho,u)$ is a smooth solution to \eqref{e:EAS}, then the monokinetic ansatz 
$f = {\rho}(x,t) \d_0(v-{u}(x,t))$, is a weak measure-valued solution to \eqref{e:VA}. At the same time, the empirical measure 
$\mu(t) = \sum_i m_i \d_{x_i(t)} \otimes \d_{v_i(t)}$ solves \eqref{e:VA} weakly iff $\{x_i,v_i\}_i$ solve \eqref{e:CS}.  This, in 
particular, establishes a direct connection between solutions of \eqref{e:CS} and the macroscopic system \eqref{e:EAS}, that is, if 
initially
\[
\mu_0 = \sum_i m_i \d_{x_i(0)} \otimes \d_{v_i(0)} \to \rho_0(x) \d_{0}(v - u_0(x))
\]
weakly$^*$,  then the corresponding evolved solutions converge for any time $t>0$ as well, i.e.,
\[
\mu(t) = \sum_i m_i \d_{x_i(t)} \otimes \d_{v_i(t)} \to \rho(x,t) \d_{0}(v - u(x,t)).
\]
The stability analysis was extended to include singular interaction forces in \cite{CCh2021} via the modulated energy method.

A vast body of work has focused on deriving macroscopic models from kinetic equations in the limit of a small Knudsen number.
A representative example is the equation
\begin{equation}\label{e:VAeintro}
\p_t f^\e + v \cdot \n_x f^\e =  \n_v( (\rho^\e_\phi v - (u^\e \rho^\e)_\phi ) f^\e ) + \frac{1}{\e} \cQ(f^\e),
\end{equation}
with $\e$ being a small dimensionless parameter representing the Knudsen number, and the force term $Q$, which depends on the physical 
regime and typically drives the system toward local equilibrium, $\cQ(f) = 0$. The relative entropy method plays a key role in analyzing 
the convergence of $f^\e$ toward equilibrium. In the case of a strong Fokker-Planck force, the collision operator takes the form
\begin{equation}\label{e:FPforce}
Q(f^\e) = \n_v \cdot (\s \n_v f^\e + (v-u^\e) f^\e),
\end{equation}
The first derivation of macroscopic limits without Cucker-Smale-type alignment was given by \cite{BV2005}, leading to the isothermal Euler 
equations with pressure $p=\s\rho$, where $\s$ is a diffusion (or noise) strength parameter. The full model \eqref{e:VAeintro} was studied 
in \cite{KMT2015}, and a more general class of environmental interaction models was considered in \cite{S-EA}. In all these settings, the 
distribution $f^\e$ relaxes to a local Maxwellian:
\begin{equation}
\mu_{\rho,u,\s} =\frac{ \rho(x,t) }{(2\pi \s)^{n/2}} e^{- \frac{|v - u(x,t)|^2}{2\s}}.
\label{e:MaxLim}
\end{equation}
Other studies explored the limit as $\s\to0$, corresponding to purely local alignment. This case, considered in \cite{KV2015} for the 
all-to-all interaction model with $\phi=\l$, exhibits weak convergence toward a monokinetic distribution due to the non-convexity of the 
underlying entropy. The full Vlasov-alignment model was addressed in \cite{FK2019}, where convergence was shown in the weak solution 
framework for non-vanishing densities. More recently, \cite{S-book,S-EA} extended these results to include vacuum states by introducing a 
Favre-type renormalization of the macroscopic velocity, which also guarantees global well-posedness of the limiting system under rough 
local alignment forces; see \sect{s:dwm} for further details.

The results mentioned so far raise an interesting question. In the monokinetic limit, how exactly does the distribution 
$f^\e$ concentrate on the Dirac mass? In other words, if $f^\e = \frac{1}{\w^n} g^\e(t,x,\frac{v-u^\e}{\w})$ for a some 
scaling parameter $\w$, does the profile stabilize $g^\e \to g$ to some smooth $g$? How to determine $g$ from its initial 
profile? These questions present interest from both the theoretical and computational perspectives as rescaling of the stiff 
term $\frac{1}{\e} \cQ(f^\e)$ can reduce computational costs, while developing a stable numerical method for 
\eqref{e:VAeintro}. 

The rescaling analysis was first implemented in the study of asymptotic behavior in \cite{ReyTan2016}, focusing on fat-tailed communication 
kernels. In this setting, the distribution $f$ is expected to converge toward a monokinetic profile centered around the total momentum. In 
\cite{ReyTan2016}, the rescaling involved a scaling parameter, $\w= \w(t,x)$, which evolved according to a transport equation along the 
native macroscopic velocity $u$, and the same velocity $u$ was used in the shift of the profile $g$.  The classical upwind finite-volume 
scheme, which preserved the physical properties of solutions, was employed to solve the rescaled system numerically. For the singular model 
in equation \eqref{e:VAeintro}, a similar rescaling was developed in \cite{CTY2018}, where the analysis not only involved the introduction 
of a rescaling ansatz but also led to the derivation of a new, rescaled system that governs the evolution of the relevant quantities. The 
new system was numerically approximated using an asymptotic preserving (AP) scheme, specifically designed to handle the stiffness of the 
problem while ensuring consistency with the original singular model. In particular, it demonstrated the non-oscillatory character of the 
family of profiles $g^\e$ and velocities $u^\e$, namely $|\n_x g^\e| \leq C_1 g^\e$ pointwise, and $\|\n_x u^\e\|_\infty <C_2$. These 
conditions cannot be proved analytically, but if assumed imply a series of stability properties of the family of solutions, and 
exponentially fast decay of the scaling parameter $\|\w\|_\infty \lesssim e^{-ct/\e}$. 

The goal of this present paper is twofold. First, we present a new monokinetic limit, where the equation is equipped with the full Fokker-
Planck force \eqref{e:FPforce} and the Knudsen and noise parameters vanish simultaneously, $\e , \s \to 0$.  We show that the monokinetic 
limit $f^\e \to  {\rho}(x,t) \d_0(v-{u}(x,t))$ still holds in the Wasserstein-2 metric and we quantify the metric itself in terms of 
$\s,\d,\e$, see \thm{t:FPlimit}. We also show that the rescaled distribution
\[
f^\e(t,x,v) = \frac{1}{\s^{n/2}} g^\e\left(t,x, \frac{v - m^\e(t,x)}{\s^{1/2}}\right)
\]
converges to the standard Gaussian $g^\e \to \mu$ --  a much more specific description than previously known. Here, the novel idea is to use 
a modulated velocity field $m^\e$ instead of the native one $u^\e$. The modulated field eliminates stiff terms from the momentum equation 
$\rho^\e u^\e$, and centers the distribution $f^\e$ more naturally, which results in the explicitly defined strong force in the 
$g^\e$-equation, $\frac{1}{\e}( \D_\xi g^\e + \n_\xi(\xi g^\e))$, pushing it towards the Gaussian. We refer to \thm{t:hydroMax} for full 
details.

Second, we implement the modulation analysis on the more traditional  Vlasov scheme, $\s = 0$. Here, the scheme allows for the complete 
elimination of stiff terms from the $g^\e$-equation, as seen in \eqref{e:ge}, and the scaling parameter $\s: = \w^\e \to 0$ vanishes 
exponentially fast with $\e$. The modulated profile stabilizes to a smooth distribution $g^\e \to g$, satisfying the transport equation
\[
\p_t g + \n_x \cdot  ( u g )   =  \n_\xi \cdot ( (\xi \cdot \n  u +  \xi \rho_\phi) g ),
\]
where $u,\rho$ are the solutions to the limiting pressureless Euler-alignment system \eqref{e:EAS}, see \thm{t:Vlasov-g}. The asymptotic 
behavior of $g$ is analytically and numerically tractable. In particular, for unidirectional flocks, we demonstrate in \sect{s:uni} that $g$ 
aggregates around the kinetic origin $\xi =0$ exponentially fast. 

It should be noted that the stability results of \thm{t:hydroMax} and \thm{t:Vlasov-g} are proved under no-oscillation conditions on the 
macroscopic field $u^\e$.  While necessary for our analysis, these conditions are less stringent than those implemented and numerically 
verified in \cite{CTY2018}.  At this moment, no unconditional stabilization of a modulated profile $g^\e$ is known for any of the schemes 
proposed.
 
\subsection{Notation} We denote $(\cdot, \cdot)_{\rho}$  the $L^2$ inner product with respect to the measure $\rho$.  We also denote 
\[
L^p(\rho) = \{ f \in \cD': \int_{\T^n} |f|^p \rho \dx <\infty\}.
\]
We will utilize Wasserstein metrics to quantify the weak convergence of measures. For the Wasserstein-1 metric on the space of probability 
measures $\cP(\O)$, where $\O$ is a Borel space, we resort to the dual (Kantorovich-Rubinshtein) definition 
\begin{align*}
W_1(\mu, \nu) = \sup_{\| \n g \|_\infty \leq 1} \big| \int_{\O} g(\o) (d\mu(\o) - d\nu(\o)) \big|.
\end{align*}
We also use the classical definition of the Wasserstein-2 metric: 
\begin{equation*}
W_2^2(\mu, \nu)= \inf_{\gamma \in \Pi(\mu, \nu)} \int_{\O \times \O} |w_1 - w_2|^2 \mbox{d}\gamma(w_1, w_2) ,
\end{equation*}
where $\Pi(\mu, \nu)$ is the set of measures with marginals $\mu$ and $\nu$, see \cite{Villani-optimal}.

\section{Monokinetic limit via Fokker-Planck scheme}
Let us consider the following model on the torus $\O = \T^n$,
\begin{equation}\label{e:FPe}
\begin{split}
\p_t f^\e + v \cdot \n_x f^\e & =\n_v( (\rho^\e_\phi v - (u^\e \rho^\e)_\phi ) f^\e) + \frac{1}{\e}[  \n_v( (v - u_\d^\e) f^\e) 
+ \s \D_v f^\e ], \\
f^\e_0(x,v) & =\mu_{\rho_0,u_0,\s}.
\end{split}
\end{equation}

The model depends on three positive parameters -- Knudsen number $\e$, noise $\s$, and resolution or mollification parameter $\d$, 
which we introduce next.  Although the solution $f^\e$ of course depends on all three parameters $f^\e = f^{\e,\d,\s}$, we keep the notation 
short since eventually $\d$ and $\s$ will be subordinated to $\e$.

Next, we properly introduce the mollification $u_\d^\e$ and discuss well-posedness of the Cauchy problem \eqref{e:FPe}.

\subsection{Density-weighted mollification}\label{s:dwm}
We start by introducing a proper mollification of the field $u$ used in \eqref{e:FPe}.  To this end, we fix a smooth mollifier with an 
algebraic decay, e.g.  
\begin{equation}\label{ }
\psi(x) = \frac{c_n}{\lan x \ran^{n+\a}}, \quad  \a>0,
\end{equation}
where $c_n$ is a normalization factor and $\lan\cdot\ran$ denote the Japanese brackets, and let 
\begin{equation}\label{ }
\psi_\d(x) = \frac{1}{\d^n} \sum_{k\in \Z^n} \psi\left(\frac{x+2\pi k}{\d}\right).
\end{equation}
Note that 
\begin{equation}\label{e:psilow}
\psi_\d(x) \gtrsim \d^\a, \quad \forall x\in \T^n.
\end{equation}

We then define  $u_\d$ to be the average of $u$ based on the following protocol,
\begin{equation}\label{e:dave}
u_\d =  \left( \frac{(u \rho) \ast {\psi_\d}}{\rho\ast {\psi_\d}} \right) \ast {\psi_\d}.
\end{equation}
Many properties and applications of this filtration are presented in \cite{S-EA,S-book}. In particular,
\begin{itemize}
\item symmetry relative to the $\rho$-weighted inner product 
\begin{equation}\label{ }
(u_\d, v)_\rho = (u, v_\d)_\rho.
\end{equation}
\item positive-semidefinite 
\begin{equation}\label{e:ball}
(u_\d,u)_{\rho} \geq (u_\d,u_\d)_{\rho}.
\end{equation}
\item approximation without any reliance on regularity of $\rho$: for any $u\in \Lip$ and any $1\leq p \leq \infty$ one has
\begin{equation}\label{e:mollest}
\| u_\d - u \|_{L^p(\rho)} \leq C \d \|u\|_{\Lip},
\end{equation}
where $C>0$ is a constant depending only on the kernel $\psi$ and $p$.
\end{itemize}
Let us note that $u_\d$ is infinitely smooth and given \eqref{e:psilow} we have the following estimate
\begin{equation}\label{ }
|\p^k u_\d| = \left|  \left( \frac{(u \rho) \ast {\psi_\d}}{\rho\ast {\psi_\d}} \right) \ast {\p^k \psi_\d} \right| 
\lesssim \d^{-k} \left\| \frac{(u \rho) \ast {\psi_\d}}{\rho\ast {\psi_\d}} \right\|_\infty 
\lesssim \d^{-k-\a} \left\| (u \rho) \ast {\psi_\d} \right\|_\infty \leq \d^{-k-\a} E^\frac12,
\end{equation}
where $E$ is the macroscopic energy
\[
E = \frac12 \int_\O |u|^2 \rho \dx.
\]
So, we obtain
\begin{equation}\label{e:uder}
\| \p^k u_\d \|_\infty \leq c \d^{-k-\a} E^\frac12.
\end{equation}

\subsection{Global well-posedness}
The alignment force in \eqref{e:FPe} is a superposition of the classical Cucker-Smale model with the excited mollified local alignment 
$\frac{1}{\e} u_\d$. For any fixed $\e,\d,\s>0$, such an alignment force falls under the class of uniformly regular models with the strength 
$\st_\rho = \rho_\phi + \frac{1}{\e}$ being bounded away from zero uniformly, see \cite{S-EA}. We also have an easier form of the constant 
thermolization parameter $\s / \e >0$ independent of the density. Hence, the results of \cite{S-EA} apply to prove global well-posedness 
from any weighted Sobolev data. To be precise, define 
\begin{equation}\label{e:Sobdef}
H^{m}_q(\domain) =  \left\{ f :  \sum_{ 2|\bk| + | \bl | \leq 2m}   \int_\domain  \jap{v}^{q - 2|\bk| - | \bl |  } 
| \p^{\bk}_{x} \p_v^{\bl} f |^2 \dv\dx <\infty \right\}.
\end{equation}
Note that the data in \eqref{e:FPe} belongs to any $H^m_q$ and is automatically thick in the sense that $\st_\rho > c_0$.  We have the following result.
\begin{theorem}
The Cauchy problem has a unique global solution $f \in C([0,T);H^m_q)$ for any $m,q\in \N$ and $T>0$. 
\end{theorem}

\subsection{Convergence of the Fokker-Planck scheme}
A traditional way of showing that $f^\e \to \mu_{\rho,u,\s}$ in the Fokker-Planck scheme is to consider the relative entropy 
\begin{equation}\label{ }
\cH(f^\e | \mu_{\rho,u,\s})  =  \s \int_\domain f^\e \log \frac{f^\e}{\mu_{\rho,u,\s}} \dv\dx
\end{equation}
which by the \CK\ dominates strong convergence
\[
\cH(f^\e | \mu_{\rho,u,\s}) \geq \s \| f^\e -  \mu_{\rho,u,\s}\|_1^1.
\]
In the vanishing noise limit $\s \to 0$, such control over the strong distance is obviously lost, and so we can no longer fully rely on the 
relative entropy method.  One component can be salvaged as $\s \to 0$, however, and that is the modulated energy 
\begin{equation}\label{ }
\en(f^\e | u) =	\frac{1}{2} \int_{\domain} | v - u|^2 f^\e \dv \dx. 
\end{equation}
We use it as a measure of deviation of $f$ from the monokinetic distribution around the limiting velocity $u$, along with the standard 
Wasserstein distances.  

\begin{theorem}\label{t:FPlimit} Let $f^\e$ be a classical solution to the Cauchy problem \eqref{e:FPe}, and let $u,\rho$ be a classical solution to \eqref{e:EAS} on a time interval. Then  the following bounds hold on $[0,T]$:
\begin{equation}\label{e:distW}
W^2_2(f^\e,f) + \en(f^\e | u) + W^2_2(\rho^\e,\rho) + W^2_1(u^\e \rho^\e,u \rho) \leq  C( \s \log\frac{1}{\s} + \e + \frac{\d}{\e} )
\end{equation}
where $C>0$ depends only on $T$ and the limiting solution $u,\rho$.
\end{theorem}

We can see from \eqref{e:distW} that the optimal resolution for the velocities is $\d = \e^2$. Then, as long as $\e ,\s \to 0$, while not 
necessarily related to each other, the corresponding solution $f^\e$ along with its macroscopic components concentrates on the limiting 
monokinetic solution. Notably, when the rates are equal $\s \log\frac{1}{\s} = \e$, which corresponds to the fastest rate, the system 
\eqref{e:FPe}  has a noise strength $\frac{\s}{\e} = \frac{1}{ \log\frac{1}{\s} } $ that is vanishing as $\s \to 0$.

\begin{proof}
To start, we introduce a hybrid version of the modulated energy which contains the Boltzmann entropy component,
\begin{equation}\label{ }
\cH_{\e} = \s \int_\domain f^\e \log f^\e \dv\dx + \en(f^\e | u).
\end{equation}
This will be the pivotal quantity that controls other distances of interest. So, let us develop an equation for it. 
First, we break it into  kinetic and macroscopic parts,
\begin{equation*}\label{}
\begin{split}
 \cH_{\e} & =  \cH^{kin}_{\e} + \cH^{macro}_{\e} \\
 \cH^{kin}_{\e} & =  \s \int_\domain f^\e \log f^\e \dv\dx + \frac{1}{2} \int_{\domain} | v |^2 f^\e \dv \dx\\
 \cH^{macro}_{\e} & =  \frac{1}{2}\int_{\O}  \rho^\e |u|^2 \dx - \int_{\O} \rho^\e u^\e \cdot u \dx.
\end{split}
\end{equation*}
In what follows, we will heavily refer to \cite{S-EA} for details. Let us start with the kinetic part:
\begin{equation}\label{e:He}
\begin{split}
\dot{ \cH}^{kin}_\e  = &- \frac{1}{\e } \int_{\domain} \frac{|\s \n_v f^\e + (v -u^\e_\d) f |^2}{f^\e}   \dv\dx 
- \frac{1}{\e } [ (u^\e_\d,u^\e)_{\rho^\e} - (u^\e_\d,u^\e_\d)_{\rho^\e}]\\
& - \int_{\domain} ( \rho^\e_\phi  v - (u^\e \rho^\e)_{\phi}) \cdot (\s \n_v f^\e + v f^\e ) \dv\dx.
\end{split}
\end{equation}
Here, the first two terms are non-positive, and integrating by parts in the third term, we obtain
\begin{equation*}\label{}
\begin{split}
\dot{ \cH}^{kin}_\e & \leq n \s \int_{\O}  \rho^\e_\phi  \rho ^\e \dx -  \int_{\domain}  \rho^\e_\phi  |v|^2 f^\e \dv\dx + \int_{\O}  (u^\e \rho^\e)_{\phi} \cdot u^\e \rho^\e \dx\\
&  \leq c \s - \int_{\O}  \rho^\e_\phi  |u^\e|^2 \rho^\e \dx + \int_{\O}  (u^\e \rho^\e)_{\phi} \cdot u^\e \rho^\e \dx.
\end{split}
\end{equation*}
Since the Cucker-Smale model is contractive, see \cite{S-EA}, the last two terms add up to a non-positive value.
Hence,
\[
\dot{ \cH}^{kin}_\e \leq c \s,
\]
while initially,
\[
\cH_\e(0) = \s \int_\O \rho_0 \log \rho_0 \dx - \frac{n}{2} \s \log (2\pi \s) \lesssim \s \log\frac{1}{\s},
\]
and so,
\[
\cH^{kin}_\e(0) = \cH_\e(0) - \cH^{macro}_\e(0)  \leq \s \log\frac{1}{\s} +c \leq C.
\]
Thus,
\begin{equation}\label{ }
\sup_{t\in [0,T]} { \cH}^{kin}_\e \leq C.
\end{equation}

By the classical inequality, \cite{Lions1994,GJV2004}, there is an absolute constant $C>0$ such that 
\begin{equation}\label{e:flogf}
\int_\domain | f \log f | \dv\dx  \leq \int_\domain f \log f \dv\dx + \frac14 \int_\domain |v|^2 f \dv\dx + C.
\end{equation}
Thus,
\begin{equation}\label{e:flogf}
\begin{split}
\s \int_\domain | f^\e \log f^\e | \dv\dx & \leq \s \int_\domain f^\e \log f^\e \dv\dx + \frac{\s}{4} \int_\domain |v|^2 f^\e \dv\dx + C\s\\
& \leq \cH^{kin}_\e + C\s \leq C.
\end{split}
\end{equation}
Consequently, we obtain a uniform bound on the energy $\cE_\e =  \frac12 \int_\domain |v|^2 f^\e \dv\dx$,
\begin{equation}\label{e:unifen}
\cE_\e = \cH^{kin}_\e  - \s \int_\domain f^\e \log f^\e \dv\dx \leq C.
\end{equation}

Next, following the steps of \cite{S-EA}, we can rewrite the kinetic entropy equation slightly differently,
\begin{equation*}\label{}
\begin{split}
\dot{\cH}^{kin}_\e  \leq -\frac{1}{\e} \cI_\e  +  c \e \en(f^\e | u^\e)    - \int_{\O}  \rho^\e_\phi  |u^\e|^2 \rho^\e \dx 
+ \int_{\O}  (u^\e \rho^\e)_{\phi} \cdot u^\e \rho^\e \dx,
\end{split}
\end{equation*}
where 
\[
\cI_\e =  \int_{\domain}  \frac{| \s \n_v f^\e+ (1+ \e \rho_\phi^\e  / 2) (v -u^\e) f^\e|^2}{f^\e}  \dv\dx.
\]
And as to the macroscopic component, we have
 \begin{equation*}\label{}
\begin{split}
\dot{\cH}^{macro}_\e  & = \int_\O  \rho^\e (u^\e - u) \cdot \n u \cdot (u - u^\e) \dx +\int_\O   \n u: \cR_\e \dx  
- \s \int_\O  \rho^\e \n \cdot u \dx \\
&+\int_\O [ \rho^\e    ( (u\rho)_\phi - u\rho_\phi ) \cdot u - \rho^\e   ((u^\e \rho^\e)_\phi- u^\e \rho^\e_\phi) \cdot u 
-  \rho^\e  ( (u\rho)_\phi - u\rho_\phi ) \cdot u^\e] \dx\\
&  - \frac{1}{\e} \int_\O \rho^\e(u^\e_\d - u^\e)\cdot u  \dx.
\end{split}
\end{equation*}
All these terms have been estimated in \cite{S-EA}:
\begin{equation*}\label{}
\begin{split}
 &\int_\O  \rho^\e (u^\e - u) \cdot \n u \cdot (u - u^\e) \dx  \leq \en(f^\e | u), \\[0.5ex]
 &\int_\O   \n u: \cR_\e \dx  \leq  \sqrt{\en(f^\e | u^\e)  \cI_\e} + \e \en(f^\e | u^\e)  
 \lesssim  \sqrt{ \cI_\e} + \e \leq - \frac{1}{2\e} \cI_\e + 2\e \\[0.5ex]
 &\int_\O  [ \rho^\e    ( (u\rho)_\phi - u\rho_\phi ) \cdot u - \rho^\e   ((u^\e \rho^\e)_\phi- u^\e \rho^\e_\phi) \cdot u 
 -  \rho^\e  ( (u\rho)_\phi - u\rho_\phi ) \cdot u^\e] \dx  \\
 & \hfill \leq   \int_{\O}  \rho^\e_\phi  |u^\e|^2 \rho^\e \dx - \int_{\O}  (u^\e \rho^\e)_{\phi} \cdot u^\e \rho^\e \dx + \en(f^\e | u) 
 + W^2_2(\rho^\e,\rho), \\[0.5ex]
 & \frac{1}{\e} \int_\O  \rho^\e(u^\e_\d - u^\e)\cdot u  \dx \leq  \frac{\d}{\e}.
\end{split}
\end{equation*}

Adding all the obtained estimates yields
\begin{equation}\label{ }
\dot{\cH}_\e \lesssim \en(f^\e | u) + \e + \frac{\d}{\e} + W^2_2(\rho^\e,\rho).
\end{equation}
As to the modulated energy, we have
\[
\en(f^\e | u) = {\cH}_\e - \s\int_\domain f^\e \log f^\e \dv\dx.
\]
Going back to \eqref{e:flogf}, let us estimate the Boltzmann term slightly differently,
\[
\s\int_\domain | f^\e \log f^\e | \dv\dx  \leq \s \int_\domain f^\e \log f^\e \dv\dx + \frac{\s}{2} \en(f^\e | u) 
+ C\s + \s \cH^{macro}_\e \leq \cH_\e + C\s,
\]
where in the last step we used that $\cH^{macro}_\e \leq C$. Thus,
\begin{equation}\label{ }
\en(f^\e | u) \lesssim \cH_\e + C\s,
\end{equation}
and consequently,
\begin{equation}\label{ }
\dot{\cH}_\e \lesssim \cH_\e + \s+ \e + \frac{\d}{\e} + W^2_2(\rho^\e,\rho).
\end{equation}

Now, let us turn to the Wasserstein distance term.  Let us consider the flow-maps associated with the $u^\e$ and $u$:
\[
\dot{X}^\e = u^\e(X^\e), \quad \dot{X} = u(X).
\]
Let us also consider the density distributions evolving according to the joint transport equation,
\[
\g_t + \n_{x_1} (u^\e(x_1) \g) + \n_{x_2} (u(x_2) \g) =0,
\]
with initial condition $\g_0 = \rho(x_1) \d_{x_1 = x_2}$. Notice that both marginals of $\g_0$ are $\rho_0$. By uniqueness, the marginals of 
$\g$ at time $t$ are $\rho^\e$ and $\rho$.  Therefore,
\[
W^2_2(\rho^\e,\rho) \leq \int_{\O^2} |x_1 - x_2|^2 \dg = \int_{\O^2} |X^\e - X|^2 \dg_0 
= \int_{\O} |X^\e(x,t) - X(x,t)|^2 \rho_0(x) \dx : = D.
\]

Let us differentiate $D$:
\begin{equation*}\label{}
\begin{split}
\dot{D} & = 2 \int_{\O} (X^\e - X)(u^\e(X^\e) - u(X)) \rho_0(x) \dx \leq D + \int_{\O} |u^\e(X^\e) - u(X)|^2 \rho_0(x) \dx \\
& \leq D +  \int_{\O} |u^\e(X^\e) - u(X^\e)|^2 \rho_0(x) \dx +  \int_{\O} |u(X^\e) - u(X)|^2 \rho_0(x) \dx\\
& \leq D+  \int_{\O} |u^\e - u|^2 \rho^\e \dx + \|\n u\|_\infty D \lesssim D+ \en(f^\e | u) \leq D + \cH_\e + C\s,
\end{split}
\end{equation*}
and arrive at the following system
\begin{equation*}\label{}
\begin{split}
\dot{\cH}_\e & \lesssim D+ \cH_\e + \s+ \e + \frac{\d}{\e} \\
\dot{D} & \lesssim D + \cH_\e + \s.
\end{split}
\end{equation*}
Remembering that $D_0 = 0$, and $\cH_\e(0) \lesssim \s \log\frac{1}{\s}$, we obtain
\[
D + \cH_\e \lesssim  \s \log\frac{1}{\s} + \e + \frac{\d}{\e}.
\]
This partially proves the claimed estimates \eqref{e:distW}. 

Let us now compare $f^\e$ directly with the monokinetic solution $f = {\rho}(x,t) \d_0(v-{u}(x,t))$. We note that initially 
\[
W_2^2(f^\e_0,f_0) = W_2^2(\mu_{\rho_0,u_0,\s}, \rho_0 \d_{u_0}) \lesssim \s.
\]
This can be seen by considering the joint distribution
\begin{equation}\label{e:g0}
\g_0 =  \d_0(x_1 - x_2)  \frac{ \rho_0(x_1) }{(2\pi \s)^{n/2}} e^{- \frac{|v_1 - u_0(x_1)|^2}{2\s}}  \d_0(v_2 - u_0(x_2))
\end{equation}
and computing the quadratic cost-function 
\[
W(0) : = \int_{\O^2 \times \R^{2n}} |\w_1 - \w_2|^2 \dg_0(\w_1,\w_2)  \sim \s.
\]

To trace the evolution of the distance at later times, we consider the following coupled Fokker-Planck equation
\begin{equation}\label{e:tg}
\begin{aligned}
\p_t \g &+ v_1 \cdot \n_{x_1} \g + v_2 \cdot \n_{x_2} \g + \n_{v_1}[ ( (u^\e \rho^\e)_\phi(x_1) - \rho^\e_\phi(x_1) v_1) \g ] \\
&+  \n_{v_2}[ (  (u \rho)_\phi(x_2) - \rho_\phi(x_2) v_2 ) \g ]
=   \frac{1}{\e} \n_{v_1}[ \s \n_{v_1} \g + (v_1 - u^\e_\d(x_1)) \g ].
\end{aligned}
\end{equation}

Let $\g$ be the solution evolving according to \eqref{e:tg} with initial condition \eqref{e:g0}. Then the marginals of $\g$ satisfy the equations of $f^\e$ and $f$ respectively. By uniqueness, those marginals are $f^\e$ and $f$ respectively.
Hence,
\[
W : = \int_{\O^2 \times \R^{2n}} |\w_1 - \w_2|^2 \dg(\w_1,\w_2)  \geq W_2^2(f^\e,f).
\]
Let $X^\e,X,V^\e,V$ satisfy the corresponding stochastic characteristic system
\begin{equation*}\label{}
\begin{split}
&\dot{X}^\e = V^\e, \quad \dot{X}  = V, \\
&\dot{V}^\e  = (u^\e \rho^\e)_\phi(X^\e) - \rho^\e_\phi(X^\e) V^\e +  \frac{1}{\e} (u^\e_\d(X^\e) - V^\e) + \sqrt{2 \s / \e} \ \dot{B}, \\
&\dot{V}  = (u \rho)_\phi(X) - \rho_\phi(X) V,\\
&X^\e(x_1,x_2,v_1,v_2,0) = x_1,  \quad  X(x_1,x_2,v_1,v_2,0) = x_2\\
&V^\e(x_1,x_2,v_1,v_2,0) = v_1, \quad V(x_1,x_2,v_1,v_2,0) = v_2,
\end{split}
\end{equation*}
where $B$ is a Brownian process. Then by the classical It\^o calculus, $\g$ is a push-forward of $\g_0$ along the stochastic characteristics, and hence,
\[
W = \E \int_{\O^2 \times \R^{2n}} (|X^\e - X|^2 + |V^\e - V|^2) \dg_0.
\]
Let us denote
\[
W_x = \E \int_{\O^2 \times \R^{2n}} |X^\e - X|^2 \dg_0, \quad W_v = \E \int_{\O^2 \times \R^{2n}} |V^\e - V|^2 \dg_0.
\]
Let us differentiate,
\[
\dot{W}_x  = 2 \E \int_{\O^2 \times \R^{2n}} (X^\e - X) \cdot  (V^\e- V)  \dg_0 \leq W_x+ W_v.
\]
As to the $v$-distance,
\begin{equation*}\label{}
\begin{split}
W_v & =   \int_{\O^2 \times \R^{2n}} |v_1 - v_2|^2 \dg \\
&  \leq  \int_{\O^2 \times \R^{2n}} |v_1 - u(x_1) |^2 \dg + \int_{\O^2 \times \R^{2n}} |u(x_1) - u(x_2)|^2 \dg 
+  \int_{\O^2 \times \R^{2n}} |u(x_2) - v_2|^2 \dg  \\
& \leq \int_{\O \times \R^{n}} |v - u(x) |^2 f^\e(x,v)\dv\dx + \|\n u\|_\infty  \int_{\O^2 \times \R^{2n}} |x_1 - x_2|^2 \dg + 0
\end{split}
\end{equation*}
where the last term vanished thanks to the monokineticity of $f$. So,
\begin{equation}\label{e:Wv}
W_v \leq  \en(f^\e|u) + C W_x \lesssim  \s \log\frac{1}{\s} + \e +\frac{\d}{\e}+ W_x.
\end{equation}
Coming back to the $W_x$-equation,
\[
\dot{W}_x \leq W_x + \s \log\frac{1}{\s} + \e+\frac{\d}{\e},
\]
and hence, by the \GL,
\[
W_x(t) \lesssim \s + \s \log\frac{1}{\s} + \e+\frac{\d}{\e} \lesssim  \s \log\frac{1}{\s} + \e + \frac{\d}{\e}.
\]
From \eqref{e:Wv} we deduce a similar bound for $W_v$. Thus, 
\[
W_2^2(f^\e,f) \leq W \lesssim  \s \log\frac{1}{\s} + \e + \frac{\d}{\e}.
\]

We also conclude the convergence of macroscopic quantities. Indeed, for the densities, this has already been proved above 
$W_2(\rho^\e, \rho) \to 0$. To see weak convergence 
\[
u^\e \rho^\e \to u \rho,
\]
as measures, the quickest way is to consider the $W_{1}$-metric. Indeed, given the results of the previous section,
\[
W_1(u^\e \rho^\e , u \rho) \leq W_1(u^\e \rho^\e , u \rho^\e)+W_1(u \rho^\e , u \rho) \leq \|u^\e \rho^\e - u \rho^\e\|_1 + W_2(\rho^\e,\rho) \leq \sqrt{\en(f^\e|u)} + W_2(\rho^\e,\rho).
\]
Thus,
\begin{equation}\label{e:uueBL}
W^2_1(u^\e \rho^\e , u \rho) \lesssim  \s \log\frac{1}{\s} + \e+ \frac{\d}{\e}.
\end{equation}
\end{proof}

\subsection{Modulation analysis}
Insofar, the convergence result stated in \thm{t:FPlimit} makes no mention about how $f^\e$ scales with $\s$. It is natural to expect, 
however, that the $\s$-rescaled  solution resembles a Gaussian profile.  This is indeed true, as will be demonstrated here under 
non-oscillation conditions on macroscopic velocities $u^\e$.  

To this end, let us consider the rescaled  solution using a new unknown modulation parameter $m^\e$, 
\begin{equation}\label{e:module2}
f^\e(t,x,v) = \frac{1}{\s^{n/2}} g^\e\left(t,x, \frac{v - m^\e(t,x)}{\s^{1/2}}\right).
\end{equation}
The modulated velocity $m^\e$ is not assumed to be the native velocity $u^\e$ as, for example, in \cite{CTY2018}. Instead, we treat it as an 
additional unknown variable along with $g^\e$.

Reading off the evolution equations for all the quantities involved, we arrive at the following system:
\begin{equation*}\label{}
\boxed{
\begin{split}
&\hspace{1.0in}\text{\bf Modulation of the Fokker-Planck-Alignment system }\\
&\p_t  \rho^\e + \n \cdot (u^\e \rho^\e)  = 0, \quad \rho_0^\e = \rho_0\\
&\p_t  (\rho^\e u^\e) + \n \cdot (\rho^\e u^\e \otimes u^\e) - \n \cdot (\rho^\e (u^\e - m^\e) \otimes (u^\e - m^\e)) 
+  \s \n_x \cdot \int_{\R^n} \xi \otimes  \xi  g^\e \dxi =  \rho^\e ((\rho^\e u^\e)_\phi - u^\e \rho^\e_\phi ),\\
&\p_t m^\e  + m^\e \cdot \n m^\e = (\rho^\e u^\e)_\phi - m^\e \rho^\e_\phi+ \frac{1}{\e}( u_\d^\e - m^\e),\quad  m^\e_0 = u_0^\e = u_0\\
&\p_t g^\e  + \n_x \cdot ( (\s \xi + m^\e) g^\e )  =  \n_\xi \cdot ( (\xi \cdot \n m^\e +  \xi \rho^\e_\phi ) g^\e) 
+ \frac{1}{\e}( \D_\xi g^\e + \n_\xi(\xi g^\e)), \quad  g^\e_0 = g_0.
\end{split}}
\end{equation*}

In this regime, we expect to show that as $\e \to 0$, the  profile converges to the local Maxwellian with the standard Gaussian in $v$,
\[
g^\e \to \mu_{\rho,0,1}  =  \frac{\rho(t,x) }{(2\pi)^{n/2}} e^{-|\xi|^2 / 2}.
\]

\begin{theorem}\label{t:hydroMax} 
Let $(u,\rho)$ be a smooth non-vacuous solution to \eqref{e:EAS}  on a time interval $[0,T]$. Let us consider the optimal resolution 
$\d = \e^2$. Suppose that
\begin{equation}\label{e:non-osc}
\sup_{\e>0} \| u^\e\|_\infty + \| \n u^\e\|_\infty + \| \n^2 u^\e\|_\infty <\infty.
\end{equation}
Then, for any $t \in [0,T]$, we have
\begin{equation}\label{ }
\cH(g^\e | \mu_{\rho,0,1}) \lesssim  \s \log\frac{1}{\s} + \e.
\end{equation}
\end{theorem}
\begin{proof}
Let us make some preliminary remarks that follow directly from the non-oscillation assumption \eqref{e:non-osc}. Since 
$\rho^\e_0 = \rho_0 > 0$ is non-vacuous, solving the continuity equation along characteristics and using that $\| \n u^\e\|_\infty <C$, we 
obtain 
\begin{equation}\label{e:rhononvac}
\rho^\e(x,t) \geq c_0 >0, \text{ for all } (x,t) \in \T^n \times [0,T].
\end{equation}
Furthermore, since $\| \n^2 u^\e\| <C$, we also have 
\begin{equation}\label{e:rhoLip}
\| \n \rho^\e\|_\infty <C.
\end{equation}
Both of these conditions immediately guarantee that the renormalization remains uniformly Lipschitz independently of $\e$ and $\d$:
\begin{equation}\label{ }
\| u_\d^\e\|_\infty + \| \n u_\d^\e\|_\infty <C.
\end{equation}

Let us now establish control over the modulated velocity $m^\e$.  First, let us read off an equation for the $L^\infty$-norm
\begin{equation}\label{ }
\p_t \|  m^\e\|_\infty \leq   \| (\rho^\e u^\e)_{\phi} \|_\infty + \frac{1}{\e}( \|u_\d^\e\|_\infty - \|  m^\e\|_\infty ) \leq c_1 
+ c_2 \frac{1}{\e } - \frac{1}{\e} \|  m^\e\|_\infty.
\end{equation}
By \GL, we have
\begin{equation}\label{e:minfty}
\|  m^\e\|_\infty \leq C.
\end{equation}

Next, we have
\[
\p_t \| \n m^\e\|_\infty \leq c \| \n m^\e\|_\infty^2 + \| \n(\rho^\e u^\e)_{\phi}\|_\infty + \|  \n \rho_\phi^\e\|_\infty 
\|  m^\e\|_\infty + c \|  \n m^\e\|_\infty  + \frac{1}{\e}(\| \n u_\d^\e\|_\infty - \| \n m^\e\|_\infty).
\]
Denoting $x =  \| \n m^\e\|_\infty$ we obtain the following inequality
\[
\dot{x} \leq c_1+ c_2 x^2 + \frac{c_3}{\e} - \frac{c_4}{\e} x.
\]
This implies that $x$ is uniformly bounded by a constant depending only on $c_1,...,c_4$. Indeed, let $C$ be large so that initially 
$x_0 \leq C$ and for a short time interval $x < 2C$.  Let $t^*\in [0,T]$ be the first time when $x(t^*) = 2C$. Then for all $t< t^*$,  we 
have
\[
\dot{x} \leq c_1 + 2c_2 C x + \frac{c_3}{\e} - \frac{c_4}{\e} x = c_1 + \frac{c_3}{\e} +(2c_2 C - \frac{c_4}{\e}) x.
\]
Integrating, and evaluating at $t= t^*$, we obtain for $\e$ small enough
\[
2C \leq C + \frac{c_1 + \frac{c_3}{\e}}{  \frac{c_4}{\e} - 2c_2 C}  \leq C + 2 \frac{c_3}{c_4}
\]
We arrive at a contradiction if $C > 2 \frac{c_3}{c_4}$. This shows that 
\begin{equation}\label{e:dm}
\| \n m^\e\|_\infty \leq 2C.
\end{equation}

Let us introduce the relative entropy into consideration:
\[
\cH(g^\e | \mu_{\rho,0,1}) =   \int_\domain g^\e \log g^\e \dv \dx + \frac12  \int_\domain |\xi|^2  g^\e - \int_\O \rho^\e \log \rho 
+ \const := \cH_1 + \cH_2 + \cH_3 + \const.
\]
Differentiating we obtain
\begin{equation*}\label{}
\begin{split}
\ddt \cH_1 & = - \int_\domain ( \xi \cdot \n m^\e \cdot \n_\xi g^\e + \st_{\rho^\e} \xi \cdot \n_\xi g^\e) \dxi \dx 
- \frac{1}{\e} \int_\domain \left( \frac{|\n_\xi g^\e|^2}{g^\e} + \xi \cdot \n_\xi g^\e \right) \dxi \dx,\\[0.5ex]
\ddt \cH_2 & =  - \int_\domain ( \xi \cdot \n m^\e \cdot \xi  + \st_{\rho^\e} |\xi|^2) g^\e \dxi\dx  
- \frac{1}{\e} \int_\domain ( \xi \cdot \n_\xi g^\e + |\xi|^2 g^\e ) \dxi \dx
\end{split}
\end{equation*}
First, let us take a look at $\cH_2$, the energy. Using \eqref{e:dm} we obtain
\[
\ddt \cH_2 \leq \frac{n}{\e}  - \frac{1}{2\e} \cH_2.
\]
This implies that $\cH_2$ remains uniformly bounded,
\begin{equation}\label{e:h2unif}
\cH_2 \leq C, \quad \forall \e>0.
\end{equation}

Now, let us go back to our computation and add up $\cH_1$ and $\cH_2$,
\[
\ddt (\cH_1 +\cH_2) =- \int_\domain ( \xi \cdot \n m^\e \cdot \n_\xi g^\e + \st_{\rho^\e} \xi \cdot \n_\xi g^\e) \dxi \dx 
- \int_\domain ( \xi \cdot \n m^\e \cdot \xi  + \st_{\rho^\e} |\xi|^2) g^\e \dxi\dx - \frac{1}{\e} \cI(g^\e),
\]
where $\cI(g^\e)$ is the Fisher information 
\[
\cI(g^\e) = \int_\domain \frac{|\n_\xi g^\e + \xi g^\e|^2}{g^\e} \dxi \dx.
\]
Given \eqref{e:dm}, the other two terms are bounded by a constant multiple of $\cH_2$, which, in view of \eqref{e:h2unif}, is uniformly 
bounded. So, we obtain
\begin{equation}\label{ }
\ddt (\cH_1 +\cH_2) \leq  C - \frac{1}{\e} \cI(g^\e). 
\end{equation}

Now, by the log-Sobolev-in-$\xi$ inequality, $\cI(g^\e) \gtrsim \cH(g^\e | \mu_{\rho^\e,0,1})$, so we arrive at
\begin{equation}\label{ }
\ddt (\cH_1 +\cH_2) \leq  C- \frac{c_2}{\e} \cH(g^\e | \mu_{\rho^\e,0,1}). 
\end{equation}

It remains to take care of $\cH_3$,
\[
\ddt \cH_3 = \int_\O \p_t \rho^\e \log \rho \dx + \int_\O \rho^\e \p_t \rho \frac{1}{\rho } \dx =  \int_\O  \rho^\e u^\e  \cdot \n \rho 
\frac{1}{\rho} \dx - \int_\O \rho^\e \n \cdot (u \rho) \frac{1}{\rho } \dx.
\]
In view of the no-vacuum assumption and the assumption on the regularity of the limiting solution, we have
\[
\ddt \cH_3 \lesssim  \int_\O  \rho^\e |u^\e|\dx + \int_\O  \rho^\e \dx \leq \sqrt{\cE_\e} + 1 \leq c.
\]

To conclude,
\begin{equation}\label{ }
\ddt \cH(g^\e | \mu_{\rho,0,1})  \leq C- \frac{c_3}{\e} \cH(g^\e | \mu_{\rho^\e,0,1}).
\end{equation}

To complete the estimate, we notice that 
\[
\cH(g^\e | \mu_{\rho,0,1}) = \cH(g^\e | \mu_{\rho^\e,0,1}) + \int_\O \rho^\e \log \frac{\rho^\e}{\rho} \dx.
\]
Since $\rho^\e$ is uniformly Lipschitz, \eqref{e:rhoLip}, and uniformly non-vacuous \eqref{e:rhononvac}, we estimate using the result of the 
previous section,
\begin{equation*}\label{}
\begin{split}
\int_\O \rho^\e \log \frac{\rho^\e}{\rho} \dx & =  \int_\O \rho^\e (\log \rho^\e - \log \rho ) \dx 
 = \int_0^1  \int_\O \frac{\rho^\e}{\th \rho^\e + (1-\th) \rho} (\rho^\e - \rho ) \dx \dth \\
 & \leq C W_1^2(\rho^\e,\rho) \leq C W_2^2(\rho^\e,\rho) \lesssim  \s \log\frac{1}{\s} + \e.
\end{split}
\end{equation*}
So,
\begin{equation}\label{ }
\ddt \cH(g^\e | \mu_{\rho,0,1})  \leq c_1 + c_2 \frac{\s \log\frac{1}{\s} }{\e} - \frac{c_3}{\e} \cH(g^\e | \mu_{\rho,0,1}).
\end{equation}
By \GL,
\[
\cH(g^\e | \mu_{\rho,0,1}) \lesssim  \s \log\frac{1}{\s} + \e,
\]
as desired. 

As to the convergence of momenta, let us observe
\begin{equation*}\label{}
\begin{split}
\| m^\e \rho^\e - u^\e \rho^\e\|_1^2 \leq  \int_\O |m^\e - u^\e|^2  \rho^\e\dx \leq \int_\domain |v - m^\e|^2 f^\e \dx \dv 
= \s \cH_2 \leq \s.
\end{split}
\end{equation*}
So, $m^\e \rho^\e \to u^\e \rho^\e$ strongly in $L^1$. Combining with \eqref{e:uueBL}, we can estimate the $W_1$-distance of modulated 
momentum to the limiting one
\begin{equation}\label{ }
W^2_{1}(m^\e \rho^\e , u \rho)  \lesssim  \s \log\frac{1}{\s} + \e.
\end{equation}
\end{proof}

\section{Vlasov scheme. Maximum principle}
Introducing the penalization force in the form of local alignment $F = \n_v( (v- u)f)$ we will examine the hydrodynamic limit as $\e \to 0$ of solutions to 
\begin{equation}\label{e:VAe}
\p_t f^\e + v \cdot \n_x f^\e = \n_v( (\rho^\e_\phi v - (u^\e \rho^\e)_\phi ) f^\e) + \frac{1}{\e}  \n_v( (v - u^\e) f^\e),
\end{equation}
where $\rho^\e$ and $\rho^\e u^\e$ are the macroscopic density and momentum of $f^\e$. It is expected that the penalization force drives the solution to monokinetic ansatz $f^\e \to f = {\rho}(x,t) \d_0(v-{u}(x,t))$, where ${\rho}, {u}$ solve the pressureless Euler Alignment system \eqref{e:EAS}.

The following theorem was proved in \cite{FK2019} with an improvement to vacuous solutions and convergence in $W_2$ in \cite{S-EA}.

\begin{theorem}\label{t:hydrolim}Let $(\rho,u)$ be a classical solution to \eqref{e:EAS} on the time interval $[0,T]$ and let $f = \rho(x,t) \d_0(v-u(x,t))$. Suppose $f_0^\e \in \Lip(\domain)$ is a family of initial conditions for \eqref{e:VA} satisfying 
\begin{itemize}
	\item[(i)] $\supp f_0^\e \ss B_R(0)$;
	\item[(ii)] $W_2(f_0^\e, f_0) \leq \e$.
\end{itemize}
Then there exists a constant $C$ such that for all $t\leq T$ one has
\begin{equation}\label{e:W1ed}
W_2(f^\e_t, f_t)   \leq C  \sqrt{\e}.
\end{equation}
\end{theorem}

The drawback of this convergence result is the lack of specificity as to how the distribution $f^\e$ approaches its limiting atomic state.  

Let us consider the following initial data
\begin{equation}\label{}
f^\e_0(x,v) = \frac{1}{\e^n} g_0 \left(x, \frac{ v- u_0}{\e} \right),
\end{equation}
where $g_0 \in C^\infty (\domain)$ is a non-negative, compactly supported  function  with $\int g_0 \dv = \rho_0$. This automatically 
fulfills the assumptions (i) and (ii) above.

Following the ideas of \cite{CTY2018,ReyTan2016} we introduce a disingularization of $f^\e$ in the form given by 
\begin{equation}\label{e:module}
f^\e(t,x,v) = \frac{1}{(\o^\e(t))^n}g^\e\left(t,x, \frac{v - m^\e(t,x)}{\o^\e(t)}\right),
\end{equation}
where $m^\e$ and $\o^\e$ are viewed as unknown modulation parameters -- velocity and scaling factor, respectively.  The main idea here is 
not to assume that $m^\e$ is the native velocity field $u^\e$. Letting it be a free unknown will enable us to find a more natural evolution 
equation for it, which does not involve any kinetic stresses, as in \cite{CTY2018}.

Plugging \eqref{e:module} into \eqref{e:VAe} and pairing the various similar terms together, we arrive at the following system 
\begin{align}\label{}
&\p_t \o^\e  = - \frac{1}{\e} \o^\e, & &\o^\e(0) = \e, \\
&\p_t m^\e + m^\e \cdot \n m^\e =   (\rho^\e u^\e)_\phi - m^\e \rho^\e_\phi  + \frac{1}{\e}( u^\e - m^\e),& & m^\e_0 = u_0, \label{e:m} \\
&\p_t g^\e + \n_x \cdot ( (\o^\e \xi + m^\e) g^\e )  =  \n_\xi \cdot ( (\xi \cdot \n m^\e +  \xi\rho^\e_\phi) g^\e) , & & g^\e_0 = g_0 ,
\label{e:ge}
\end{align}
with the scaling parameter having an explicit form
\begin{equation}\label{e:oe}
\o^\e = \e e^{-t / \e}.
\end{equation}

This system is coupled with the continuity and momentum equations for the pair $\rho^\e, u^\e$, which can be written in terms of $g^\e$ as 
follows
\begin{equation}\label{e:macroe}
\begin{split}
&\p_t \rho^\e + \n \cdot (u^\e \rho^\e)  = 0\\
&\p_t (\rho^\e u^\e) + \n \cdot (\rho^\e u^\e \otimes u^\e) - \n \cdot (\rho^\e (u^\e - m^\e) \otimes (u^\e - m^\e)) +  \n_x \cdot \cR_\e 
=  \rho^\e ((\rho^\e u^\e)_\phi - u^\e \rho^\e_\phi ),
\end{split}
\end{equation}
where
\[
\cR_\e  =(\o^\e)^2 \int_{\R^n} \xi \otimes  \xi  g^\e(t,x,\xi) \dxi.
\]

One unfortunate feature of the system is that the momentum of $m^\e$ is not conserved. So, it would be impossible to rewrite the $m$-
equation in divergence form. However, we can do it with an additional small linear in $m$ term. Indeed, multiplying 
\eqref{e:m} with $\rho^\e$ and noting that 
\[
\int_\O \xi g^\e(t,x,\xi)\dxi = \frac{\rho^\e}{\o^\e}(u^\e - m^\e),
\]
we obtain
\[
\rho^\e \p_t m^\e + \rho^\e m^\e \cdot \n m^\e = \p_t(\rho^\e  m^\e) + \n (\rho^\e  m^\e \otimes m^\e) 
+ \o^\e m^\e \n_x \cdot \int_\O \xi g^\e(t,x,\xi)\dxi.
\]

Let us gather the obtained equations in one full system.
\begin{equation*}\label{}
\boxed{
\begin{split}
&\hspace*{1.2in}\text{\bf Modulation of the  Vlasov-Alignment system }\\
& \o^\e  = \e e^{-t / \e},\\
&\p_t  \rho^\e + \n \cdot (u^\e \rho^\e)  = 0,\\
&\p_t  (\rho^\e u^\e) + \n \cdot (\rho^\e u^\e \otimes u^\e) - \n \cdot (\rho^\e (u^\e - m^\e) \otimes (u^\e - m^\e)) 
+  (\o^\e)^2 \n_x \cdot \int_{\R^n} \xi \otimes  \xi  g^\e \dxi =  \rho^\e ((\rho^\e u^\e)_\phi - u^\e \rho^\e_\phi ),\\
&\p_t (\rho^\e  m^\e) + \n (\rho^\e  m^\e \otimes m^\e) + \o^\e m^\e \n_x \cdot \int_\O \xi g^\e\dxi  = \rho^\e ((\rho^\e u^\e)_\phi 
- m^\e \rho^\e_\phi ) + e^{-t/\e} \int_\O \xi g^\e\dxi, \\
&\p_t  g^\e + \n_x \cdot ( (\o^\e \xi + m^\e) g^\e ) =  \n_\xi \cdot ( (\xi \cdot \n m^\e + \xi \rho^\e_\phi ) g^\e).
\end{split}}
\end{equation*}
 
We expect that in the limit as $\e\to 0$ the module converges to 
\[
\rho^\e \to \rho, \quad m^\e,u^\e \to u, \quad g^\e \to g
\]
where $(\rho,u,g)$ solves
\begin{equation}\label{e:limitginviscid}
\boxed{
\begin{split}
&\text{\bf Pressureless EAS } \text{\bf with the limiting profile equation}\\
&\p_t \rho+ \n\cdot ({u}{\rho}) = 0, \\
&\p_t u + u \cdot \n {u}  = (u \rho)_\phi - u \rho_\phi,\\
&\p_t g + \n_x \cdot  ( u g )  =  \n_\xi \cdot ( (\xi \cdot \n  u +  \xi \rho_\phi) g ).
\end{split}
}
\end{equation}
\smallskip

\begin{theorem}\label{t:Vlasov-g}
Suppose the macroscopic family of velocities is uniformly smooth,
\begin{equation}\label{e:gradu}
 \sup_{\e>0} \| \n u^\e\|_\infty + \| \n^2 u^\e\|_\infty <\infty.
\end{equation}
Suppose also that the initial profile $g_0$ has a compact support and finite order of contact with $0$, i.e. there exist $\a\in (0,1)$ and 
$C>0$ such that 
\begin{equation}\label{e:order}
|\n_x g_0 | \leq C g_0^\a.
\end{equation}
Let $[0,T]$ be a time interval of existence of a smooth solution $(u,\rho,g)$ to the system \eqref{e:limitginviscid}. Under these 
assumptions the module \eqref{e:module} remains uniformly smooth in $\e$ in the following sense
\begin{align}
\sup_{\e>0, t\in [0,T]} \|g^\e\|_{W^{1,\infty}} & <\infty,\\
\sup_{\e>0, t\in [0,T]}\|  m^\e\|_{W^{2,\infty}} & <\infty,\label{e:gradm} 
\end{align}
and  converges to its limiting profile weakly on $[0,T]$ with the following rates
\begin{align}\label{ }
W_1(\rho^\e, \rho)  & \leq C  \sqrt{\e}, \\
W_1(u^\e \rho^\e, u \rho) + W_1(m^\e \rho^\e, u \rho) & \leq C  \sqrt{\e}, \\
W_1(g^\e,g) &\leq \e^{\frac{\a}{4}}.
\end{align}
\end{theorem}
\begin{proof}
Let us discuss regularity first. 

Since the initial support of $g_0$ is bounded, then so are the supports of $f_0^\e$ uniformly,
\begin{equation}\label{ }
\supp f^\e_0 \ss \O \times B_R(0), \qquad \forall \e>0.
\end{equation}
Next, looking at the $v$-characteristic
\begin{equation*}\label{}
\begin{split}
\dot{X^\e} & = V^\e \\
\dot{V^\e}& = \int_\O  \phi(X^\e-{X^\e}')({V^\e}' - V^\e) f^\e_0(\o') \domega' 
+  \frac{1}{\e}\frac{\int_{\R^n} (v - V^\e) f^\e (X^\e,v,t)\dv}{\rho^\e(X^\e)}, 
\end{split}
\end{equation*}
where $\o = (x,v)$, one can see by the maximum principle that the support of $f^\e$ will remain within its initial bounds in $v$. So,
\begin{equation}\label{ }
\supp f^\e \ss \O \times B_R(0), \qquad \forall \e, t >0.
\end{equation}
This implies, in particular, that $u^\e$ remains bounded as well
\[
|u^\e(x,t)| = \frac{\left| \int_{\R^n}v f^\e (x,v,t)\dv \right|}{\int_{\R^n} f^\e (x,v,t)\dv} \leq \max |\supp f^\e|.
\]
Looking back at the $m$-equation, \eqref{e:m} we can see that $m^\e$ remains uniformly bounded:
\begin{equation}\label{ }
\p_t \|  m^\e\|_\infty \leq   \|  \st_{\rho^\e} \ave{u^\e}\|_\infty + \frac{1}{\e}( \|u^\e\|_\infty - \|  m^\e\|_\infty ) \leq c_1 
+ \frac{1}{\e}(c_2 - \|  m^\e\|_\infty).
\end{equation}
By \GL, we have
\begin{equation}\label{e:minfty}
\|  m^\e\|_\infty \leq c.
\end{equation}

As to its gradients, we have
\[
\p_t \| \n m^\e\|_\infty \leq c \| \n m^\e\|_\infty^2 + \| \n( u^\e \rho^\e)_\phi \|_\infty + \|  \n \rho^\e_\phi\|_\infty \|  m^\e\|_\infty 
+  \|  \n m^\e\|_\infty  + \frac{1}{\e}(\| \n u^\e\|_\infty - \| \n m^\e\|_\infty).
\]
Denoting $x =  \| \n m^\e\|_\infty$, \eqref{e:gradu}, and \eqref{e:minfty} we have the following inequality
\[
\dot{x} \leq c_1+ c_2 x^2 + \frac{c_3}{\e}(1  - x).
\]
We know that initially $x_0$ remains uniformly bounded, say $x_0 \leq c_0$. We claim that $x \leq 2c_0+1$ on the given time interval 
$[0, T]$.  Indeed,  let $t^*$ be a hypothetical first time when $x(t^*) = 2c_0+1$. So, $x(t) < 2c_0+1$, for $t<t^*$. We then have
\[
\dot{x} \leq c_1+ c_2 (2c_0+1)^2 + \frac{c_3}{\e}(1  - x),
\]
and by \GL,
\[
x(t) \leq c_0 + \frac{c_1+ c_2 (2c_0+1)^2 + \frac{c_3}{\e}}{c_3 / \e} < c_0 + 1+\frac12 c_0,
\]
provided $\e$ is small enough.  Hence, $t^*$ does not exist.  Doing a similar computation for the second-order derivatives and relying on 
the assumption \eqref{e:gradu}, we obtain a uniform bound on the second gradient. 

By the transport nature of the equation for $g^\e$, \eqref{e:ge}, we can see that at any time $g^\e$ is a push-forward of the initial 
$g^\e_0$ along the characteristics
\begin{equation*}\label{}
\begin{split}
\dot{X^\e} & = \o^\e \Sigma^\e + m^\e(t,X^\e), \\
\dot{\Sigma^\e} & = - \Sigma^\e \cdot \n m^\e(t,X^\e) - \rho^\e_\phi(X^\e) \Sigma^\e.
\end{split}
\end{equation*}
With \eqref{e:gradm} at hand, we can see that the supports of $g^\e$ remain uniformly bounded in $\e$ as characteristics emanating from the 
initial support of $g_0$ will remain bounded. 

Let us notice that the Jacobian of the characteristic map is given by
\[
\det \n_{\o} (X^\e,\Sigma^\e) (t,\o) = \exp \left\{  -n \int_0^t \rho^\e_\phi(X^\e(s,\o)) \ds \right\}.
\]
Then for any $t>0$,
\begin{equation}\label{ }
g^\e(t, X^\e(t,\o),\Sigma^\e(t,\o)) = g_0(\o) \exp \left\{ n \int_0^t \rho^\e_\phi(X^\e(s,\o)) \ds \right\}.
\end{equation}
Inverting the flow and noting that $(X^\e,\Sigma^\e)$ and $\rho^\e_\phi$ are Lipschitz implies $g^\e \in \Lip(\domain)$ uniformly.

Let us now discuss the limit $g^\e \to g$.  Since we will seek to justify this limit in the weak sense, let us fix $h$, $\Lip(h) \leq 1$, 
and follow the characteristics
\[
\int_\domain h(\o) (g^\e - g) \domega = \int_\domain [h(X^\e(t,\o),\Sigma^\e(t,\o)) - h(X(t,\o),\Sigma(t,\o))] g_0(\o)\domega,
\]
where $(X,\Sigma)$ is the characteristics of the $g$-equation:
\begin{equation}\label{e:SXsys}
\begin{split}
\dot{X} & =  u(t,X), \\
\dot{\Sigma} & = - \Sigma \cdot \n u(t,X) - \rho_\phi(X) \Sigma.
\end{split}
\end{equation}
Continuing,
\[
\left| \int_\domain h(\o) (g^\e - g) \domega  \right| \leq  \int_\domain [|(X^\e(t,\o) -X(t,\o)|+ |\Sigma^\e(t,\o)) 
- \Sigma(t,\o)|] g_0(\o)\domega.
\]
It will be more convenient to work with the square metrics, so we denote
\begin{equation*}\label{}
\begin{split}
W_x  &= \int_\domain |X^\e(t,\o) -X(t,\o)|^2 g_0(\o)\domega \\
W_\xi  &= \int_\domain |\Sigma^\e(t,\o) -\Sigma(t,\o)|^2 g_0(\o)\domega
\end{split}
\end{equation*}
According to the above, we have
\begin{equation}\label{e:Wgg}
W_1(g^\e,g)  \leq \sqrt{W_x} +\sqrt{W_\xi}.
\end{equation}

Our goal is now to show that $W_x , W_\xi \to 0$. We will rely on the following two estimates established in \cite{S-EA}:
\begin{align}\label{e:W1ed}
W_1(\rho^\e, \rho)  & \leq C  \sqrt{\e}, \\
 \int_{\domain} | v - u(x,t)|^2 f^\e(x,v,t) \dx \dv & \leq C  \e.
\end{align}
The latter quantity dominates the macroscopic relative entropy, so we have
\begin{equation}\label{e:mre}
 \int_{\domain} | u^\e - u|^2 \rho^\e \dx  \leq C  \e.
\end{equation}

Let us compute derivatives 
\begin{equation*}\label{}
\begin{split}
\dot{W}_x =  2 \int_\domain (X^\e -X)\cdot (\o^\e \Sigma^\e+m^\e(X^\e) - u(X))  g_0(\o)\domega
\end{split}
\end{equation*}
Here and throughout, we will be using the fact that $\Sigma^\e$ characteristics emanating from the support of $g_0$ are uniformly bounded. 
Thus,
\begin{equation*}\label{}
\begin{split}
\dot{W}_x  & \leq c_1 \e + W_x + \int_\domain |m^\e(X^\e) - u(X)|^2  g_0(\o)\domega \\ 
& \leq c_1 \e + W_x +\int_\domain |m^\e(X^\e) - u^\e(X^\e)|^2  g_0(\o)\domega  +\int_\domain |u^\e(X^\e) - u(X^\e)|^2  g_0(\o)\domega \\
&+ \int_\domain |u(X^\e) - u(X)|^2  g_0(\o)\domega \\
&\leq c_1 \e + c_2 W_x +\int_\domain |m^\e(x) - u^\e(x)|^2  g^\e(t,\o)\domega  +\int_\domain |u^\e(x) - u(x)|^2  g^\e(t,\o)\domega \\
&\leq c_1 \e + c_2 W_x +\int_\O |m^\e - u^\e|^2  \rho^\e\dx  +\int_\O |u^\e - u|^2  \rho^\e\dx \\
\intertext{and according to \eqref{e:mre},}
&\dot{W}_x\leq (c_1+C) \e + c_2 W_x +\int_\O |m^\e - u^\e|^2  \rho^\e\dx .
\end{split}
\end{equation*}
The mean deviation $\int_\O |m^\e - u^\e|^2  \rho^\e\dx$ is bounded by the internal energy centered at $m^\e$:
\[
 \int_\O |m^\e - u^\e|^2  \rho^\e\dx \leq \int_\domain |v - m^\e|^2 f^\e \dx \dv : = \eta_\e,
\]
which will be treated separately. We have obtained
\begin{equation}\label{ }
\dot{W}_x \lesssim \e +  W_x + \eta^\e.
\end{equation}

Let us continue to $W_\xi$ and using uniform bounds \eqref{e:gradu}, \eqref{e:gradm},
\begin{equation*}\label{}
\begin{split}
\frac12 \dot{W}_\xi  &=   \int_\domain (\Sigma^\e -\Sigma)\cdot ( - \Sigma^\e \cdot \n m^\e(X^\e) - \rho^\e_\phi(X^\e) \Sigma^\e 
+\Sigma \cdot \n u(X) + \rho_\phi(X) \Sigma ) g_0(\o)\domega \\
&=   \int_\domain (\Sigma^\e -\Sigma)\cdot ( \Sigma- \Sigma^\e) \cdot \n m^\e(X^\e) g_0(\o)\domega 
+  \int_\domain (\Sigma^\e -\Sigma)\cdot  \Sigma \cdot (\n m^\e(X^\e) - \n m^\e(X)) g_0(\o)\domega\\
&+ \int_\domain (\Sigma^\e -\Sigma)\cdot  \Sigma \cdot (\n m^\e(X) - \n u(X)) g_0(\o)\domega\\
&- \int_\domain |\Sigma^\e -\Sigma|^2 \rho^\e_\phi(X^\e)  g_0(\o)\domega - \int_\domain (\Sigma^\e -\Sigma)\cdot \Sigma (\rho^\e_\phi(X^\e) 
- \rho^\e_\phi(X))g_0(\o)\domega  \\
&- \int_\domain (\Sigma^\e -\Sigma)\cdot \Sigma (\rho^\e_\phi(X) - \rho_\phi(X))g_0(\o)\domega\\
& \lesssim W_\xi + \sqrt{W_\xi W_x} + \int_\domain (\Sigma^\e -\Sigma)\cdot  \Sigma \cdot (\n m^\e(X) - \n u(X)) g_0(\o)\domega 
+ \sqrt{W_\xi W_x} + \int_\O |(\rho^\e - \rho)_\phi| \rho \dx.
\end{split}
\end{equation*}
Let us look into the two unresolved terms. The last one is simply bounded by a constant multiple of $W_1(\rho^\e,\rho) \lesssim \sqrt{\e}$ 
simply because $\phi$ is a smooth kernel and the convolution with $\phi$ is dominated by the $W_1$-metric.

The middle unresolved term will be estimated by relating the velocities of the gradients. To this end, we observe
\[
\n m^\e(X) - \n u(X) =  (\n_x X)^{-\top}\n_x (m^\e(X) - \n u(X))
\]
so, integrating by parts,
\begin{equation*}\label{}
\begin{split}
\int_\domain (\Sigma^\e -\Sigma)\cdot  \Sigma \cdot (\n m^\e(X) - \n u(X)) g_0(\o)  \domega 
= - \int_\domain [m^\e(X) -  u(X)] \n_x[  (\n_x X)^{-\top} (\Sigma^\e -\Sigma)\cdot  \Sigma  g_0] \domega.
\end{split}
\end{equation*}
The moved gradient will fall on either $\xi$-characteristics, which is uniformly bounded given \eqref{e:gradm}, or on $(\n_x X)^{-\top} $ 
which also is bounded by regularity of the $(X,\Sigma)$-system, or it falls on $g_0$. So, in the first two cases, we have an estimate by 
\begin{equation*}\label{}
\begin{split}
\lesssim & \int_\domain |m^\e(X) -  u(X)|g_0 \domega = \int_\O |m^\e - u| \rho \dx\\
 = & \int_\O |m^\e - u| (\rho - \rho^\e) \dx + \int_\O |m^\e - u| \rho^\e \dx \\
\lesssim & W_1(\rho^\e, \rho) + \int_\O |m^\e - u^\e| \rho^\e \dx + \int_\O |u^\e - u| \rho^\e \dx \lesssim \sqrt{\e} + \sqrt{\eta_\e}.
\end{split}
\end{equation*}
And in the last case, we have by  \eqref{e:order}
\begin{equation*}\label{}
\begin{split}
\lesssim &  \int_\domain |m^\e(X) -  u(X)||\n_x g_0| \domega \leq \int_\domain |m^\e(X) -  u(X)|g_0^\a \domega \\
 \leq & |\supp g_0|^{1-\a}\left( \int_\domain |m^\e(X) -  u(X)|^{\frac{1}{\a} }g_0\domega \right)^{\a}.
\end{split}
\end{equation*}
Noting that $1/\a >1$, we simply reduce the power down to $1$ by using uniform boundedness of velocities, and the resulting term is bounded 
by the one we estimated in the previous step. So, we obtain
\[
\lesssim (\e + \eta_\e)^{\frac{\a}{2}}.
\]
To summarize, we have arrived at the following system
\begin{equation}\label{e:WW}
\begin{split}
\dot{W}_x &\lesssim  W_x + \e + \eta_\e, \hspace{0.7in} W_x(0) = 0 \\
 \dot{W}_\xi & \lesssim W_\xi + W_x +  (\e + \eta_\e)^{\frac{\a}{2}},  \quad W_\xi(0) = 0.
\end{split}
\end{equation}

It remains to address the internal energy $\eta_\e$. Let us notice that 
\[
\eta_\e = (\o^\e)^2 \int_\domain |\xi|^2 g^\e \dx \dxi.
\]
We have
\[
\ddt  \int_\domain |\xi|^2 g^\e \dx \dxi = -  \int_\domain(\xi \cdot \n m^\e \cdot \xi +  |\xi|^2 \rho^\e_\phi )g^\e\dx \dxi \leq C \int_\domain |\xi|^2 g^\e\dx \dxi.
\]
So, the integral remains uniformly bounded and hence,
\[
\eta_\e \leq \e^2.
\]
The system \eqref{e:WW} now reads
\begin{equation}\label{e:WW2}
\begin{split}
\dot{W}_x &\lesssim  W_x + \e , \hspace{0.6in} W_x(0) = 0 \\
 \dot{W}_\xi & \lesssim W_\xi + W_x +  \e ^{\frac{\a}{2}},  \quad W_\xi(0) = 0,
\end{split}
\end{equation}
and by \GL, 
\[
W_x + W_\xi \leq C \e^{\a/2}.
\]
Plugging into \eqref{e:Wgg} finishes the result. 
\end{proof}

\begin{remark} Let us notice that the macroscopic momentum of $g$ satisfies
\begin{equation}\label{}
\begin{split}
&{\rho} m = \int_{\R^n} \xi g(t,x,\xi)\dxi \\
&\p_t {\rho} + \n\cdot (u {\rho}) = 0, \\
&\p_t m + m  \cdot \n u +u \cdot \n m = - {\rho}_\phi m.
\end{split}
\end{equation}
So, if initially $g_0$ is well-centered, i.e., $m_0 \equiv 0$, then it will remain centered for all times by uniqueness.
\end{remark}

\section{Qualitative behavior of the modulated solution to the Vlasov Scheme}\label{s:uni}
This section is devoted to studying the qualitative features of the modulated solution of the Vlasov scheme, \eqref{e:limitginviscid}.  
For uni-directional flocks introduced in \cite{LS-uni1}, the $\xi$-characteristics will collapse to zero thanks to the additional 
conservation law, the $e$-quantity, see below and \cite{TT2014,CCTT2016}.  This analysis relies on bounding $e$ from below, which is not 
possible in more general situations. We provide a numerical illustration in the one-dimensional (1-D) case to reinforce these findings.  

\subsection{Characteristic squeezing for uni-directional flocks}
%The $x$-characteristics of \eqref{e:limitginviscid} are given by 
%	\begin{align}
%		\ddt {X_x} &= \bu,  & X_1(s,s,x,\xi) &= x, 
%	\end{align}
The characteristics for the $g$-equation in \eqref{e:limitginviscid} are given by \eqref{e:SXsys}.
%\begin{align} 
%	\label{eqn:xi_characteristics}
%	\ddt \Sigma &= - \bxi \cdot \n \bu - \bxi \rho_{\phi}, & \bxi(s,s,x,\xi) &= \xi.  
%\end{align}
The flock is called uni-directional if the velocity field points in one direction,
\[
U(t,x) = u(t,x) \bar{d}, \qquad \bar{d} \in \cS^{n-1}, u: \R^+ \times \Omega^n \mapsto \R.
\]
The maximum principle implies that the velocity remains uni-directional for all time.  
By rotation invariance, we can assume that $\bar{d}$ is the unit vector $e_1$, i.e. $U(t,x) = (u(t,x), 0, \dots, 0)$, so 
\eqref{e:limitginviscid} becomes 
\begin{equation}\label{e:unisys}
	\begin{split}
		&\p_t \rho+ \partial_{x_1} ({u}{\rho})  = 0, \\
		&\p_t u + u \partial_{x_1} {u}   = (u \rho)_\phi - u \rho_\phi,\\
		&\p_t g + \partial_{x_1}  ( u g )   =  \n_\xi \cdot ( (\xi \cdot \n U +  \xi \rho_\phi) g ),
	\end{split}
\end{equation}
where 
\[
	\xi \cdot \n U = (\xi \cdot \n u, 0, \dots, 0). 
\]
In the following, let $c>0$ denote a constant which can change line to line. 
\[
	e = \partial_{x_1} u + \rho_{\phi}, \qquad \partial_t e + \partial_{x_1} (e u) = 0 . 
\] 
It has been established in \cite{ST2} in 1-D and \cite{LS-uni1} for unidirectional flocks, that if $e_0$ and $\rho_0$ are smooth and bounded 
away from zero, then they will remain so uniformly in time, and  $\| \n u \|_{\infty} < C$.  

Expanding the $\xi$-characteristics of \eqref{e:unisys} for each component, we obtain 
\begin{equation*}
	\begin{split}
		\ddt \Sigma_1 &= -\big( \partial_{x_1} u + \rho_{\phi} \big) \Sigma_1 + \partial_{x_2} u \Sigma_2  + \dots + \partial_{x_n} u \Sigma_n \\
		\ddt \Sigma_i &= - \Sigma_i \rho_{\phi}, \quad i=2,\ldots,n.
	\end{split}
\end{equation*}
Since $0 < c < e, \rho$ and $\| \n u \|_{\infty} < C$, we obtain 
\begin{equation*}
	\begin{split}
		\ddt \Sigma_1 &\leq  -c \Sigma_1 + C \big( \Sigma_2  + \dots +  \Sigma_n \big), \\
		\ddt \Sigma_i &= - c \Sigma_i , \quad i=2,\ldots,n.
	\end{split}
\end{equation*}
By \GL, we conclude that the $\xi$-characteristics exponentially decay to zero. The exponential squeezing of characteristics is observed 
numerically in Section \ref{sec:numerics}. 

\subsection{Symmetry points of the Euler alignment system} 
In the 1D numerical solution, we observe that $\rho$ and $u$ are symmetric about these zeros of $u$-- specifically, $\rho$ is even and $u$ 
is odd about these zeros. The following computation proves it.  

Let $x^*$ be a point such that $u(x^*) = 0$.  WLOG, set $x^* = 0$.  
Let $r(t,x) = \rho(t,-x)$ and $v(t,x) = -u(t, -x)$.  For brevity, we will omit the time parameter $t$ as it does not affect the 
computations. 
Using the identities 
\[
	\n r(x) = - (\n \rho)(-x), \qquad \n v (x) =  (\n u) (-x),
\]
we obtain
\begin{align*}
	\partial_t r(x)
		&= \partial_t \rho(-x) =  - (\n \cdot  (u \rho)) (-x)    \\
		&= - \big( (\n \cdot u) \rho + u \cdot \nabla \rho \big)(-x) \\
		&= - \big( (\n \cdot  v) r + v \cdot \n r \big)(x) \\
		&= - \n \cdot \big( v r ) (x) ,
\end{align*}
and 
\begin{align*}
	\partial_t v(x) 
		&= -\partial_t u(-x) \\
		&=  \big( u \cdot \n u \big)(-x) - \big( (u \rho)_{\phi} - u \rho_{\phi} \big) (-x) \\
		&= - \big( v \cdot \nabla v \big) (x) - \big( (u \rho)_{\phi} - u \rho_{\phi} \big) (-x)
\end{align*}
The convolution integral is given by 
\begin{align*}
	\big( (u \rho)_{\phi} - u \rho_{\phi} \big) (-x) 
		&= \int_{\Omega} \phi(-x-y) (u(y) - u(-x)) \rho(y) dy  \\
		&=  \int_{\Omega} \phi(z-x) (u(-z) - u(-x)) \rho(-z) dz \\
		&=  - \int_{\Omega}  \phi(z-x) (v(z)  - v(x) )  r(z) dz \\
		&=  - \int_{\Omega}  \phi(x-z) (v(z)  - v(x) )  r(z) dz  = -\big( (v r)_{\phi} - v r_{\phi} \big) (x) . 
\end{align*}
We used the symmetry of the kernel in the last line. 
We obtain 
\[
	\partial_t v(x) + v \cdot \nabla v =   (v r)_{\phi} - v r_{\phi} .
\]
Thus, if the initial density and velocity are even and odd, respectively, then by uniqueness they $\rho = r$ satisfy the same equation and 
that $u =v$.

\subsection{Numerical illustration of exponential squeeze}
\label{sec:numerics}
To illustrate the analytical findings described in this section, we compute numerical solutions to the limiting inviscid system 
\eqref{e:limitginviscid} using a constant interaction kernel, $\phi(|x|) = 1$, and the following initial data:
\begin{align*}
	\rho_0(x) = 1, \qquad g_0(x,\xi) =  \frac{1}{\sqrt{2\pi \sigma}} \exp{\{ -\xi^2 / (2\sigma) \}}, \qquad \sigma =\frac{1}{10}. 
\end{align*}
These simulations are not intended as definitive numerical studies but serve only as visual support for the theoretical behavior observed 
earlier.

We examine two cases for the initial velocity field—one possessing odd symmetry around its zeros and another lacking such a structure—to 
highlight the influence of symmetry on the evolution dynamics, particularly in the formation of density concentration (squeezing).

In the first case, we consider a velocity profile with odd symmetry about its zeros:
\[
u_0(x) = \frac{1}{3\pi} \cos(2\pi x). 
\]
The corresponding simulation results are displayed in Figure 1. The top row shows the evolution of $g(x,\xi,t)$, while the bottom plots 
display the evolution of $\rho(x,t)$ and $u(x,t)$. As expected, the odd symmetry of $u_0$ is preserved: the velocity remains odd and the 
density remains even with respect to the symmetry center. Squeezing of $g$ develops symmetrically in regions of steep negative velocity 
gradient, leading to the formation of sharp, localized structures. This coherent behavior underscores the stabilizing influence of symmetry 
in the dynamics.

In the second case, we take an initial velocity profile, which does not exhibit odd symmetry:
\[
u_0(x) = \frac{1}{4} \Big( \frac{1}{2\pi} \sin(2\pi x) + \frac{1}{3\pi} \cos(2\pi x)  + \frac{1}{2\pi} \sin(4\pi x)  
+ \frac{1}{4\pi} \cos(4\pi x)  + \frac{1}{8\pi} \sin(6\pi x)  + \frac{1}{5\pi} \cos(6\pi x)  \Big) .
\]
The results are shown in Figure \ref{num_soln:g_nonsym}. In this case, while squeezing of $g$ still occurs due to localized regions of 
strong negative velocity gradient, the evolution is more irregular and lacks the organized structure seen in the symmetric case. The 
locations, shapes, and intensities of the density peaks vary more widely, and the behavior is less predictable over time.

Again, we emphasize that these numerical simulations are not exhaustive but are meant to illustrate and support the analytical insights 
regarding symmetry preservation, structure formation, and the onset of density concentration.

\begin{remark}
The initial data satisfies $e_0 = \partial_x u_0 + (\rho_0)_{\phi} > 0$ and therefore the solutions exist globally in time. 
In the numerical solution, a spike appears in the density, but this is a numerical artifact due to the fact that it lies in a 
finite-dimensional space.  
\end{remark} 
\begin{remark}
Our solutions are computed on the torus, so we obtain similar qualitative behavior when the kernel is bounded below, e.g., 
$\phi(|x|) = \frac{1}{(1 + |x|^2)^{\beta/2}}$. 
We chose $\phi = 1$ for computational simplicity. 
\end{remark} 
\begin{figure}[!h]
  \begin{center}
          \begin{tabular}{cccc | cc }
              \includegraphics[width=0.25\linewidth]{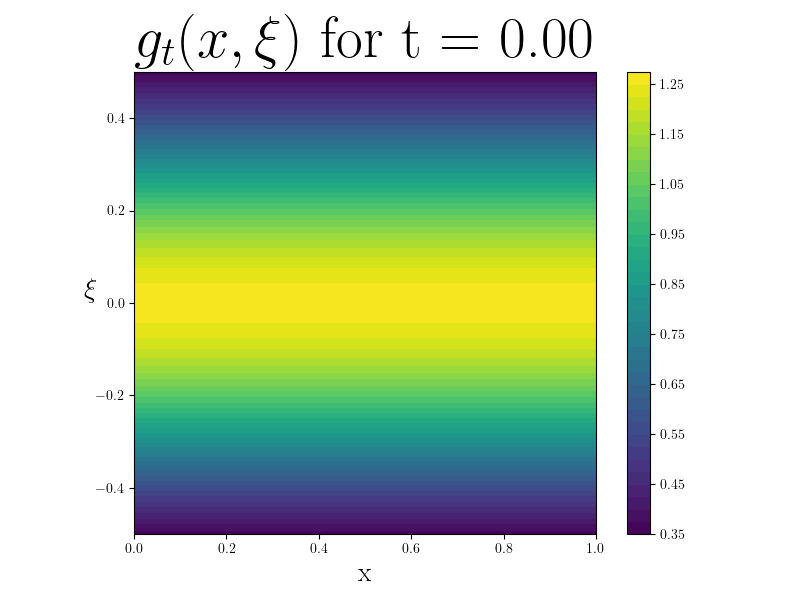} 
              \includegraphics[width=0.25\linewidth]{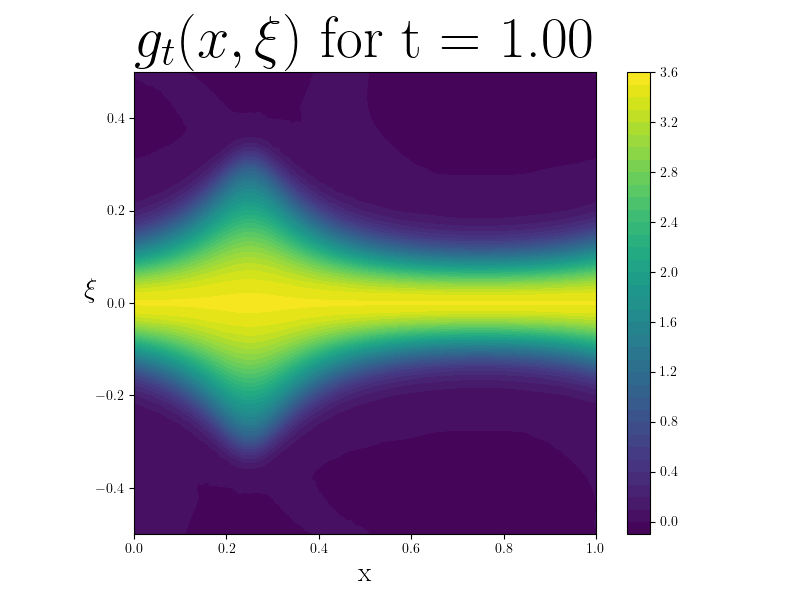}
              \includegraphics[width=0.25\linewidth]{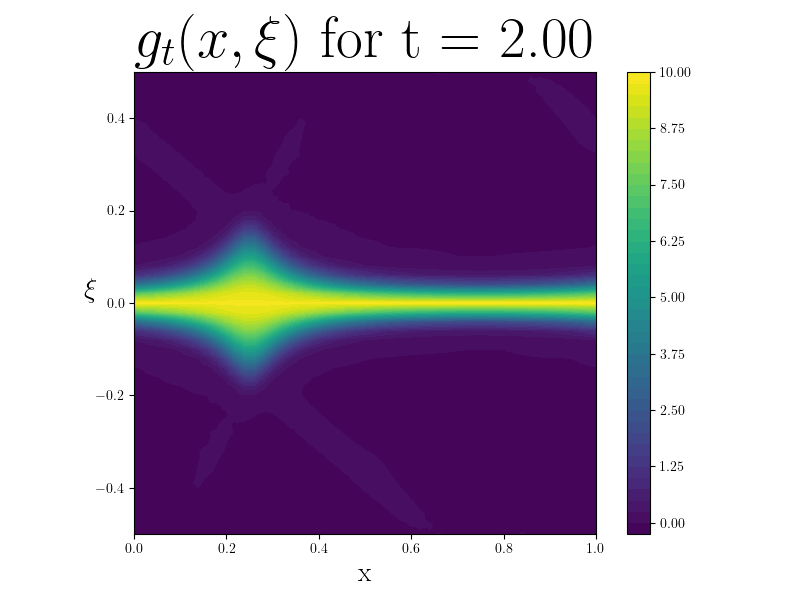} 
              \includegraphics[width=0.25\linewidth]{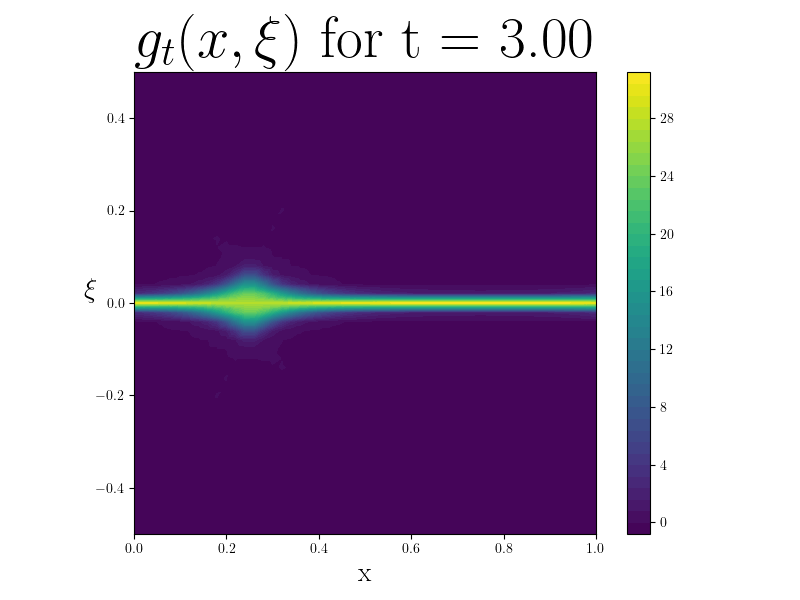}  \\
              \includegraphics[width=0.5\linewidth]{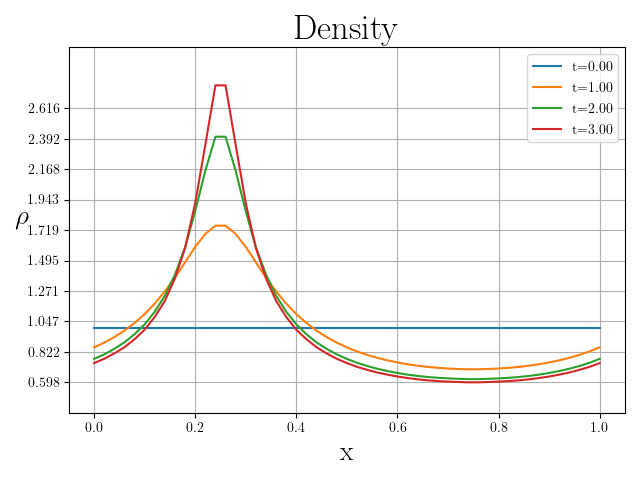}
              \includegraphics[width=0.5\linewidth]{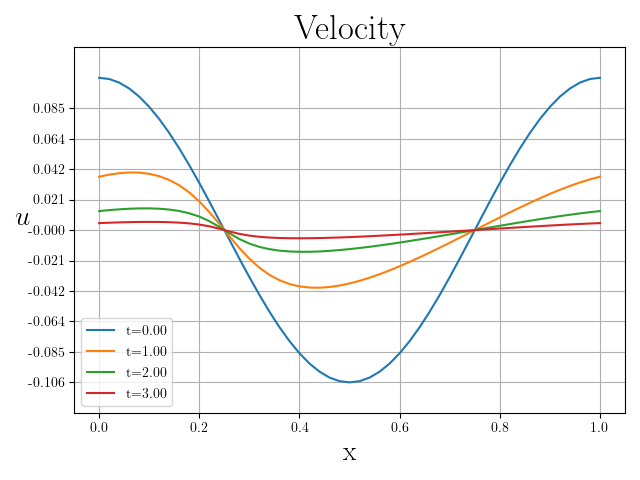}
          \end{tabular}
  \end{center} 
\caption{Numerical solution of $(\rho, u, g)$ when $u_0$ has an odd symmetry around its zeros.  The contour plot of $g$ is shown at four 
equally spaced time points.  Time moves left to right, and the leftmost image is $g_0(x,\xi)$.   For the density and velocity, the initial 
profile is shown in blue, and the remaining time points are depicted in the colorbar.  \label{num_soln:g}}
\end{figure}
\begin{figure}[!h]
  \begin{center}
          \begin{tabular}{cccc | cc}
              \includegraphics[width=0.25\linewidth]{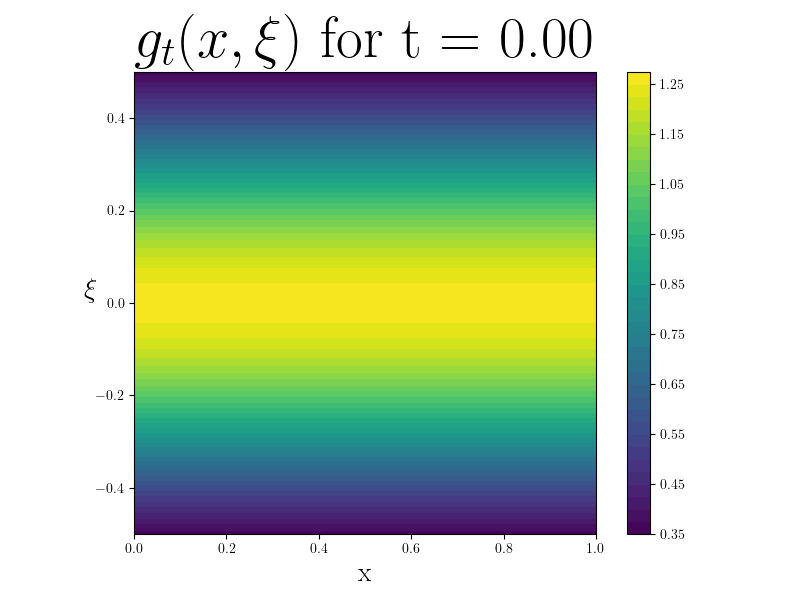}
              \includegraphics[width=0.25\linewidth]{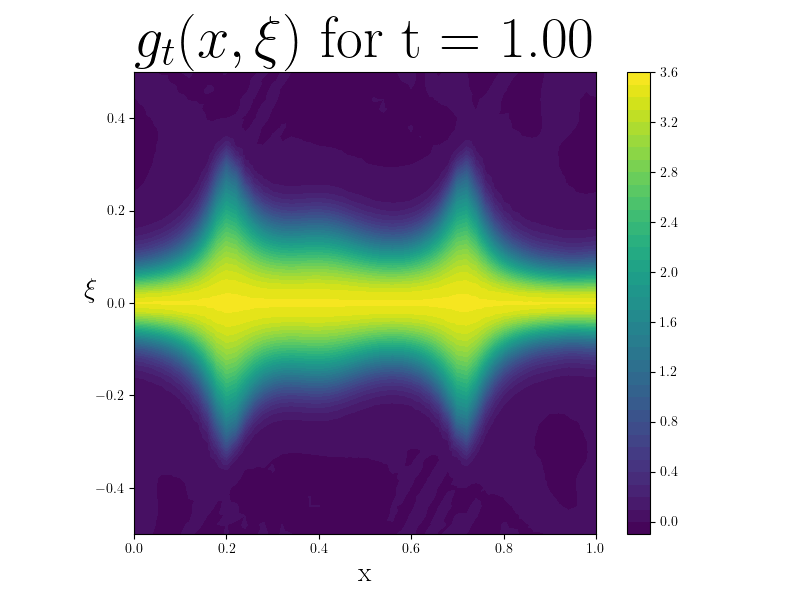}
              \includegraphics[width=0.25\linewidth]{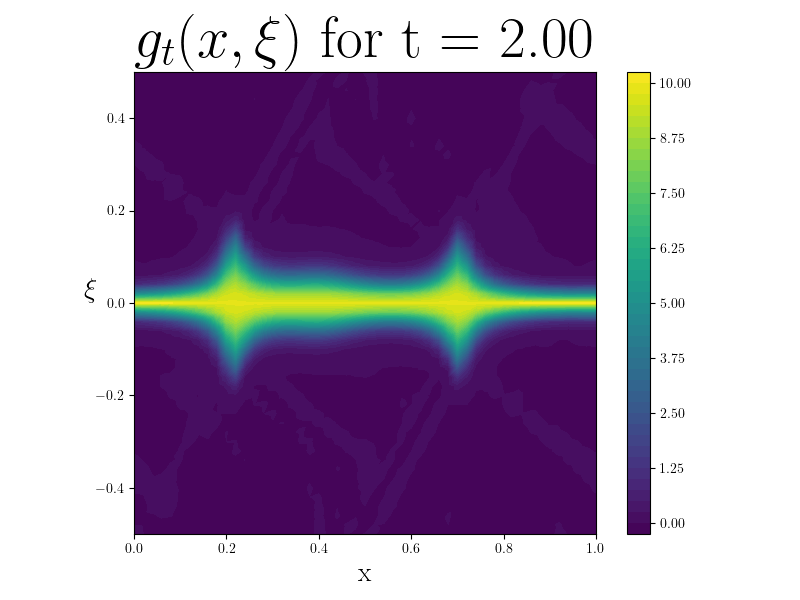}
              \includegraphics[width=0.25\linewidth]{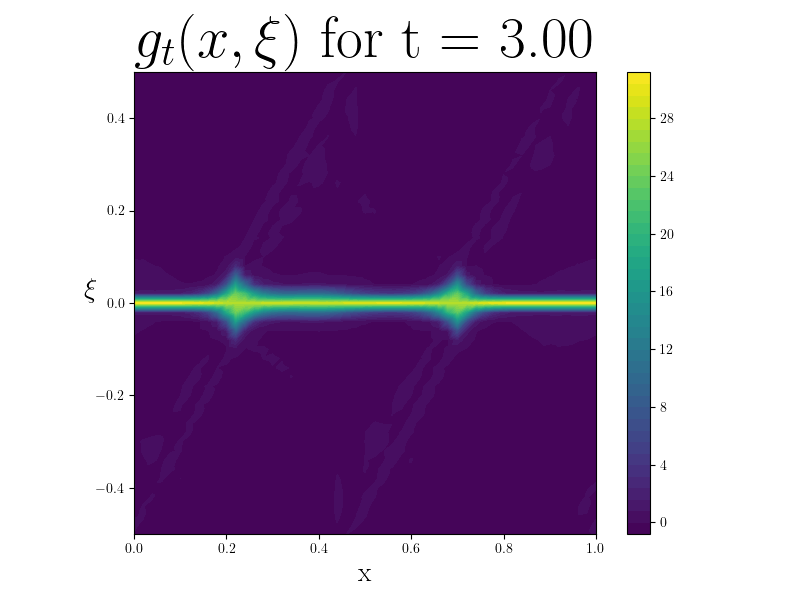} \\
              \includegraphics[width=0.5\linewidth]{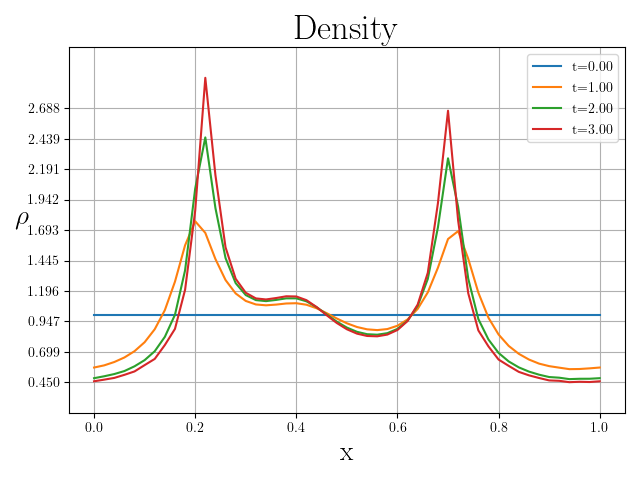}
              \includegraphics[width=0.5\linewidth]{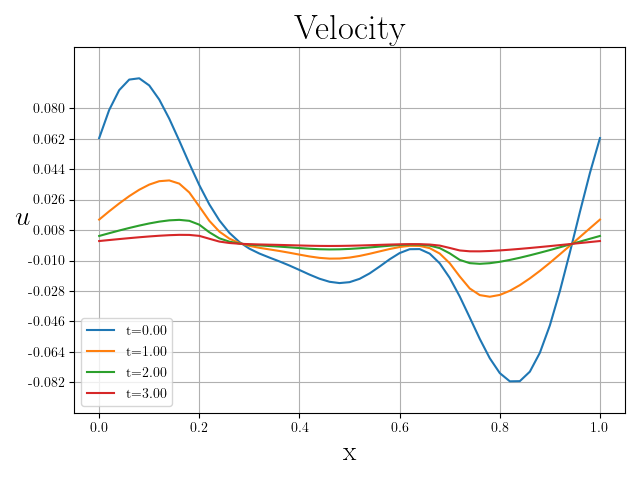} 
          \end{tabular}
  \end{center} 
\caption{Numerical solution of $(\rho, u, g)$ when $u_0$ has no symmetry around its zeros.  The contour plot of $g$ is shown at four equally 
spaced time points.  Time moves left to right, and the leftmost image is $g_0(x,\xi)$.   For the density and velocity, the initial profile 
is shown in blue, and the remaining time points are depicted in the colorbar.  \label{num_soln:g_nonsym}}
\end{figure}

\bibliographystyle{alpha}
\bibliography{collective,shvydkoy,kinetic,analysis&pde,extra_references}

\end{document}